%% file: hbnDemazure.tex
\documentclass[10pt]{amsproc}
\providecommand{\showcomments}{1} 
\usepackage{pfl}
\title{Transmission permutations and Demazure products in Hurwitz--Brill--Noether theory}
\pflauthorinfo
\date{\today}

\usepackage{import} 
\usetikzlibrary{decorations.pathreplacing}

\newcommand{\asp}{\operatorname{ASP}}
\newcommand{\Inv}{\operatorname{Inv}}

\newcommand{\ess}{\operatorname{Ess}}
\newcommand{\sa}{s_\alpha}
\newcommand{\sai}{s_{\alpha^{-1}}}
\newcommand{\sbe}{s_\beta}

\newcommand{\st}{s_\tau}
\newcommand{\sti}{s_{\tau^{-1}}}
\newcommand{\tpql}[1][{}]{\tau^{p_{#1},q_{#1}}_{\cL_{#1}}}
\newcommand{\ca}{{\chi_\alpha}}
\newcommand{\cb}{{\chi_\beta}}

\newcommand{\spql}{s^{p,q}_\cL}
\newcommand{\sk}[1]{\sigma^k_{#1}}
\newcommand{\skm}{\sk{m}}

\newcommand{\ts}[1]{\widetilde{\Sigma}_{#1}}

\begin{document}

\maketitle

\begin{abstract}
A line bundle on a curve with two marked points can be special in many ways, as measured by the global sections of all of its twists by these points. All of this information is conveniently packaged into a permutation, which we call the transmission permutation. We prove that when twice-marked curves are chained together, these permutations are composed via the Demazure product; in reverse, bundles with given permutation can be enumerated via reduced decompositions of a permutation. 
This paper demonstrates the utility of transmission permutations by giving a short derivation of the basic dimension bounds of both classical Brill--Noether theory and Hurwitz--Brill--Noether theory in a unified framework. The difference between the two cases derives from taking permutations in either symmetric groups or affine symmetric groups.
\end{abstract}



\section{Introduction}
\label{sec:intro}

How do you measure how special a line bundle on a curve is? The classical answer is of course the number that was originally called the index of speciality, $h^1(C,\cL)$; or one might just as well say $h^0(C,\cL)$, since it carries the same information. Brill--Noether theory bids you take the two numbers together, since their product is the expected codimension in $\Pic(C)$ of equally special line bundles. Denoting the (projective) dimension of the complete linear series of $\cL$ by $r$, the degree of $\cL$ by $d$, and the genus of $C$ by $g$, this product is $h^0(\cL)\, h^1(\cL) = (r+1)(g-d+r)$.

If the curve has a marked point $p$, finer distinctions are possible. You may ask for the vanishing orders of $\cL$, which measure inflection; this amounts to asking not just for $h^0(C,\cL)$ but a function $f(n) = h^0(C,\cL(-np))$. To be fully informed, one may as well allow negative $n$. This information is neatly packaged in a combinatorial datum, which is named the \emph{Weierstrass partition} in \cite{pflChains}. A slick way to form it is to plot all the points $\left\{\left(h^0\left(C, \cL\left(-np\right)\right), h^1\left(C, \cL\left(-np\right)\right)\right): n \in \ZZ\right\}$ in $\NN^2$; they form a ``staircase path'' tracing out a Young diagram, and \emph{voil\`a}, a partition. This isn't just a gimmick; the combinatorics of the partition knows about interesting geometry; its size generalizes the number $(r+1)(g-d+r)$ above and tells the expected codimension of equally special bundles, and the number of set-valued Young tableaux of content $\{1,\cdots,g\}$ tells the algebraic Euler characteristic of the Brill--Noether variety \cite{cpEuler,actKclasses}.

But why stop at one marked point? How should one measure how special a line bundle on a twice-marked curve $(C,p,q)$ is? One answer is to track two functions $h^0(C, \cL(-np))$ and $h^0(C, \cL(-nq))$. This has been the standard approach going back to the generalized Brill--Noether theorem of \cite{eh86}, and generalizes nicely to three or more marked points. The story is particularly nice for two marked points, which is the most for which the story behaves well in positive characteristic. In place of a partition, one can still build a combinatorial datum, a \emph{skew tableau}. The size of the skew tableaux tells expected codimension, set-valued tableaux tell the Euler characteristic, and the corners of the tableaux inform you of the singular locus for general $(C,p,q)$ \cite{cop,cpEuler,actKclasses,teixidorBNTwo}. 

Along with \cite{pflVersality}, this paper aims to promote a different, richer combinatorial datum for twice-marked curves. This datum remembers more: a two-variable function $f(a,b) = h^0(C, \cL(ap-bq))$ that knows not just how $p,q$ are inflected, but how they interact. The datum is a permutation $\ZZ \to \ZZ$, which we call the \emph{transmission permutation}. 
Like the Weierstrass partition, this permutation is not a mere bookkeeping device: its combinatorics knows interesting geometric information. In place of the number $(r+1)(g-d+r)$ or the size of a (skew) tableau, the \emph{number of inversions} of the permutation predicts the codimension of equally special divisors, and reduced words take the place of tableaux in enumerative questions. We also get a bonus: by considering permutations in the extended affine symmetric group, we learn not about general points in $\cM_{g,2}$, but general points in a Hurwitz space (as explained below). 
Intimately linked to this story is a curious associative operation on permutations related to tropical matrix multiplication called the Demazure product.

The aim of this paper is modest: to demonstrate the utility of this construction with a short unified proof of the existence of Brill--Noether and Hurwitz--Brill--Noether general curves. See also the partner paper \cite{pflBriefTropBN}, which develops similar notions in the tropical context.

\begin{rem}
    The permutations we use in this paper differ from those used in \cite{pflVersality}, in that increasing, rather than decreasing permutations are the most generic; this is because we describe $h^0(C,\cL(ap-bq))$, rather than $h^0(C, \cL(-ap-bq))$, via permutations. The choice in this paper allows a cleaner use of the Demazure product, and reflects the fact that, when chaining curves together, a generic degree $g$ line bundle has no effect, and thus corresponds to multiplying by the identity permutation.
\end{rem}

\subsection{The transmission permutation}
Let $(C,p,q)$ be a twice-marked smooth curve. If $\cL$ is a line bundle on $C$, there is a unique permutation $\tau = \tpql : \ZZ \to \ZZ$ characterized by
\begin{eqnarray}
\label{eq:tauh0}
h^0(C, \cL(ap-bq)) &=& \#\{n \geq b: \tau(n) \leq a\}, \text{ and }\\
\label{eq:tauh1}
h^1(C, \cL(ap-bq)) &=& \#\{n < b: \tau(n) > a\}, \text{ for all } a,b \in \ZZ.
\end{eqnarray}
%
\begin{figure}
\begin{tabular}{cc|cc}
\begin{tikzpicture}[scale=0.2]
\def\a{0}
\def\b{0}
\fill[lightgray] (-0.5+\b, \a+0.5) rectangle (6,-6);
\fill[lightgray] (-0.5+\b, \a+0.5) rectangle (-6,6);

\draw[->] (-6,0) -- (6,0) node[right] {$n$};
\draw[->] (0,-6) -- (0,6) node[above] {$\iota_3(n)$};

\foreach \m in {-2,...,5} {
\node [fill=black, circle, inner sep=1pt] at (\m,-3+\m) {};
}
\end{tikzpicture}
&
\begin{tikzpicture}[scale=0.2]
\def\a{-2}
\def\b{3}
\fill[lightgray] (-0.5+\b, \a+0.5) rectangle (6,-6);
\fill[lightgray] (-0.5+\b, \a+0.5) rectangle (-6,6);

\draw[->] (-6,0) -- (6,0) node[right] {$n$};
\draw[->] (0,-6) -- (0,6) node[above] {$\iota_3(n)$};

\foreach \m in {-2,...,5} {
\node [fill=black, circle, inner sep=1pt] at (\m,-3+\m) {};
}
\end{tikzpicture}

&
\begin{tikzpicture}[scale=0.2]
\def\a{0}
\def\b{0}
\fill[lightgray] (-0.5+\b, \a+0.5) rectangle (6,-6);
\fill[lightgray] (-0.5+\b, \a+0.5) rectangle (-6,6);

\draw[->] (-6,0) -- (6,0) node[right] {$n$};
\draw[->] (0,-6) -- (0,6) node[above] {$\tau(n)$};

\foreach \m in {-5,...,-2} {
\node [draw=black, fill=white, circle, inner sep=1.3pt] at (\m,\m-1) {};
}
\foreach \m in {3,...,6} {
\node [draw=black, fill=white, circle, inner sep=1.3pt] at (\m,\m-1) {};
}
\node [fill=black, circle, inner sep=1.3pt] at (-1,1) {};
\node [fill=black, circle, inner sep=1.3pt] at (0,-1) {};
\node [fill=black, circle, inner sep=1.3pt] at (1,0) {};
\node [fill=black, circle, inner sep=1.3pt] at (2,-2) {};
\end{tikzpicture}
&
\begin{tikzpicture}[scale=0.2]
\def\a{-2}
\def\b{2}
\fill[lightgray] (-0.5+\b, \a+0.5) rectangle (6,-6);
\fill[lightgray] (-0.5+\b, \a+0.5) rectangle (-6,6);

\draw[->] (-6,0) -- (6,0) node[right] {$n$};
\draw[->] (0,-6) -- (0,6) node[above] {$\tau(n)$};

\foreach \m in {-5,...,-2} {
\node [draw=black, fill=white, circle, inner sep=1.3pt] at (\m,\m-1) {};
}
\foreach \m in {3,...,6} {
\node [draw=black, fill=white, circle, inner sep=1.3pt] at (\m,\m-1) {};
}
\node [fill=black, circle, inner sep=1.3pt] at (-1,1) {};
\node [fill=black, circle, inner sep=1.3pt] at (0,-1) {};
\node [fill=black, circle, inner sep=1.3pt] at (1,0) {};
\node [fill=black, circle, inner sep=1.3pt] at (2,-2) {};
\end{tikzpicture}
\\
$h^0(C,\cL) = 4$
&
$h^0(C, \cL(-2p-3q)) = 0$
&
$h^0(C,\cL) = 3$
&
$h^0(C,\cL(-2p-2q)) = 1$
\\
$h^1(C,\cL) = 0$
&
$h^1(C, \cL(-2p-3q)) = 1$
&
$h^1(C,\cL) = 1$
&
$h^1(C,\cL(-2p-2q)) = 3$
\end{tabular}
\caption{Two examples of transmission permutations, and what they say about $h^0$ and $h^1$ of twists $\cL(ap-bq)$ of $\cL$. See Examples~\ref{eg:iota} and \ref{eg:bitangent}.
}
\label{fig:taus}
\end{figure}
This permutation $\tau$ is called the \emph{transmission permutation} of $\cL$ on $(C,p,q)$. Its existence and uniqueness are established in \Cref{lem:tauExists} below; the reader may find it profitable to prove this now in order to digest the definition. This permutation conveniently packages substantial geometric information about $\cL$ beyond vanishing orders, including the presence of nodes, bitangents, and other features linking $p$ to $q$ (see \cite[Figure 1]{pflVersality} and \Cref{eg:bitangent} below).
Transmission permutations need not have finite length, but they do satisfy the following finiteness condition. 

\begin{defn} \label{def:asp}
Denote by $\asp$ the group of bijections $\tau : \ZZ \to \ZZ$ such that $n \tau(n) > 0$ for all but finitely many integers $n$. We call such permutations \emph{almost-sign-preserving}. The \emph{shift} of an almost-sign-preserving permutation $\tau$ is the number $\chi_\tau$ defined by $$\chi_\tau = \# \{ n \geq 0: \tau(n) < 0 \} - \# \{ n < 0 : \tau(n) \geq 0 \}.$$
\end{defn}

The definition of $\tau$ shows that $\chi_{\tau} = \chi(C, \cL(-p)) = d - g$, where $d$ is the degree of $\cL$ and $g$ is the genus of $C$. So the shift divides $\asp$ into cosets corresponding to the degree of the line bundle in question.

\begin{eg} \label{eg:iota}
Let $\tau = \iota_{d-g}$, where $\iota_{d-g}(n) = n - (d-g)$. This permutation has shift $d-g$. Equations~\eqref{eq:tauh0} and \eqref{eq:tauh1} say that $h^0(C,\cL(ap-bq)) = \max\{0, a - b + d - g + 1\}$ and $h^1(C,\cL(ap-bq)) = \max\{0, b - a - d + g - 1\}$. This is illustrated in the first two panels of \Cref{fig:taus}. By Riemann--Roch, this is equivalent to saying that $\deg \cL = d$ and every twist $\cL(ap-bq)$ has either $h^0 = 0$ or $h^1 = 0$. In other words, every twist $\cL(ap-bq)$ is nonspecial. Thus $\tau = \iota_{d-g}$ when $\cL$ is a \emph{very general} line bundle of degree $d$ on a genus $g$ curve. So $\iota_{d-g}$ can be seen as the ``generic'' transmission permutation in degree $d$.
\end{eg}

The values of $\tau$ admit a concrete geometric interpretation, as follows. Say that a line bundle $\cL$ \emph{has a node at $p,q$} if $|\cL(-p)| = |\cL(-q)|$, i.e.\ if every global section vanishing at $p$ also vanishes at $q$, and vice versa. Then one can check that for all $a,b \in \ZZ$,
\begin{equation} \label{eq:node}
    \tau(b) = a \quad \text{if and only if} \quad \cL(ap-bq) \; \text{has a node at $p,q$, but no base point at $p$ or $q$.} \end{equation}
With practice, this interpretation allows one to read the transmission permutation $\tau$ from the geometry of $C$, as mapped to projective space by $\cL$. Here is an illustrative example.

\begin{eg} \label{eg:bitangent}
Suppose that $C$ is a smooth plane quartic curve with distinct marked points $p,q$, and $\cL = \omega_C$. Let $\tau$ be the transmission permutation of $\cL$ on $(C,p,q)$. Riemann--Roch implies that $\cL(p+q)$ is base-point-free for all $p,q$, while $\cL(p)$ has a base point at $p$ for all $p$, and similarly for $\cL(q)$ at $q$. Therefore $\tau(-1) = 1$. The reader should carefully note the sign conventions here, which are a bit counterintuitive at first. This much is true without any additional hypotheses on $p,q$. 

Assume now the special case that $p,q$ are the two intersections of $C$ with one of the $28$ bitangent lines in the canonical embedding. That is, $\cL \cong \cO(2p+2q) \cong \omega_C$.
Then $\cL(-p-q)$ has a unique global section up to scale, and that section vanishes to order $2$ at both $p$ and $q$ when regarded as a section of $\cL$. The same is true of $\cL(-2p)$ and $\cL(-2q)$; in each case the unique section (up to scale) corresponds to the bitangent line. It follows that each of $\cL(-p)$, $\cL(-q)$, and $\cL(-2p-2q)$ has a node at $p,q$, but not a base point at either. Therefore $\tau(0) = -1$, $\tau(1) = 0$ and $\tau(2) = -2$.

These four values of $\tau$ are illustrated with filled circles in \Cref{fig:taus}, together with the corresponding values of $h^0$ and $h^1$. In fact, for a general plane quartic $C$ and a choice of bitangent, one can also deduce that $\tau(n) = n-1$ for all $n \leq -2$ and $n \geq 3$; these remaining values are shown in \Cref{fig:taus} with unfilled circles.
\end{eg}

Since the most generic transmission permutation is strictly increasing, we may think of the set of inversions of $\tau$ as measuring the extent to which $\cL$ and its twists at $p,q$ are special. For example, in \Cref{eg:bitangent}, there are five inversions. This is because our choices are special in two ways: $\cL \cong \omega_C$ (a codimension-$3$ condition since $\dim \Pic^4(C) = 3$), and $p,q$ span a bitangent line (a codimension-$2$ condition). It is no coincidence that these codimensions add up to $3+2 = 5$, the number of inversions of $\tau$.

Note in particular that, for any choice of $(C,p,q)$ and $\cL$, if $r = h^0(C,\cL) - 1$ then
 \[\#\{(m,n): m < 0 \leq n, \tau(n) \leq 0 < \tau(m)\} = h^0(C,\cL)\, h^1(C,\cL) = (r+1)(g-d+r),\] so the expected codimension from Brill--Noether theory is a lower bound on the number of inversions.

\subsection{Transmission loci} \label{ssec:transmission}
In the other direction, any $\tau \in \asp$ defines a subvariety of $\Pic^{\chi_\tau + g}(C)$. Define set-theoretically
\begin{eqnarray*}
W^\tau(C,p,q) &=& \Big\{ [\cL] \in \Pic^{\chi_\tau + g}(C):\ h^0(C,\cL(ap-bq)) \geq \#\{n \geq b: \tau(n) \leq a \} \text{ for all } a,b\in \ZZ \Big\},\\
&=& \Big\{ [\cL] \in \Pic^{\chi_\tau + g}(C):\ h^1(C,\cL(ap-bq)) \geq \#\{n < b: \tau(n) > a \} \text{ for all } a,b\in \ZZ \Big\}.
\end{eqnarray*}
We call this the \emph{$\tau$-transmission locus}. The fact that these two equations are equivalent comes from Riemann--Roch and Equation~\eqref{eq:duality} below. The formulation via a bound on $h^1$ is better-suited for a family of curves, as explained in \Cref{rem:h1}. That remark and surrounding discussion also explain how to give $W^\tau$ a natural scheme structure using Fitting ideals. But for readability, here and elsewhere, we will focus on set-theoretic definitions.

We are particularly interested in \emph{extended $k$-affine permutations}.
\begin{defn}
Let $k=0$ or $k\geq 2$ be an integer.
Denote by $\ts{k} \le \asp$ the group of permutations $\alpha$ such that $\alpha(n+k) = \alpha(n) +k$ for all $n \in \ZZ$. In particular, we let $\ts{0} = \asp$ (this notation differs slightly from that of \cite{pflBriefTropBN}). For $k \geq 2$, these are called the \emph{extended affine symmetric groups}.
\end{defn}
The word ``extended'' is present because when $k \geq 2$ the ``affine symmetric group'' is the subgroup of permutations such that $\sum_{n=0}^{k-1} (\alpha(n)-n) = 0$, or equivalently $\ca = 0$.

We call a twice-marked curve $(C,p,q)$ \emph{$k$-torsion} if $kp \sim kq$ as divisors, and say that $k$ is the \emph{torsion order} if it generates $\{n \in \ZZ: np \sim nq\}$. The definition of $\tpql$ shows that if $(C,p,q)$ is $k$-torsion, $\tpql \in \ts{k}$ for all $\cL$. The converse is also true; if $\tau^{p,q}_{\cO_C} \in \ts{k}$, then it follows that $h^0(C, \cO_C) = h^0(C, \cO_C(kp-kq)) =1$, so $kp \sim kq$.

An \emph{inversion} of $\alpha$ is a pair $(m,n)$ such that $m < n$ and $\alpha(m) > \alpha(n)$, and two inversions $(m,n), (m',n')$ are called $k$-equivalent if $m'-m = n'-n \equiv 0 \pmod{k}$.
The complexity of such permutations is measured by $\inv_k(\alpha)$, the number of $k$-equivalence classes of inversions; $\inv_0(\alpha)$ is simply the number of inversions.
Note that ``extended $0$-affine'' and ``$0$-torsion'' are vacuous conditions, but we allow them to state results in a uniform way.

Studying these transmission loci provides a route to studying the classical Brill--Noether varieties, as well as the more recently studied \emph{splitting loci} of Hurwitz--Brill--Noether theory. In particular,

\begin{enumerate}
\item Given $g,r,d$, there is a permutation $\gamma^r_{d-g}$ such that, for $C$ a genus $g$ curve, the Brill--Noether locus $W^r_d(C)$ is equal to $W^{\gamma^r_{d-g}}(C,p,q)$, regardless of $p,q$. The number of inversions of $\gamma^r_{d-g}$ is $(r+1)(g-d+r)$, the expected codimension of $W^r_d(C)$. 
\item Given any $k\geq 2$ and splitting type $\vec{e} \in \ZZ^k$, there is a permutation $\gamma_{\vec{e}}$ such that for any $(C,p,q)$ with $kp \sim kq$, the splitting locus $W^{\vec{e}}(C,kp)$ is equal to $W^{\gamma_{\vec{e}}}(C,p,q)$. The number of inversions of $\gamma_{\vec{e}}$ is equal to $u(\vec{e})$, the expected codimension of $W^{\vec{e}}(C,kp)$.
\end{enumerate}

See \Cref{sec:specialPerms} for terminology and the precise statements. The construction of these permutations is not novel; $\gamma_{\vec{e}}$ appears in slightly different form in \cite[Theorem 1.4]{larsonLarsonVogt}, for example. The novelty in the present paper is placing them in the context of transmission loci, and offering the Demazure product as a useful device for inductive arguments.

\subsection{The Demazure product}

The fundamental tool in our argument is the \emph{Demazure product} on $\asp$. This is an associative product $\star$ with several nice characterizations, whose properties are developed in detail in \cite{pflDemazure} and summarized in \Cref{sec:demazure} here. Briefly, $\alpha \star \beta$ can be obtained by decreasing $\alpha, \beta$ in Bruhat order until a ``reduced product'' is obtained, and then multiplying them; it is also characterized by a type of matrix multiplication over the min-plus semiring. The fundamental observation on which this entire paper turns is that, if $(C_1, p_1, q_1), (C_2, p_2, q_2)$ are two twice-marked chains (or smooth curves), and $q_1$ is glued to $p_2$ to form $(X, p_1, q_2)$, then for any line bundle $\cL$ on $X$ restricting to $\cL_1, \cL_2$ on $C_1, C_2$,
\begin{equation}
\label{eq:demazure}
\tau^{p_1, q_2}_\cL = \tpql[1] \star \tpql[2].
\end{equation}
This is proved in \Cref{cor:chainT}. In reverse, transmission loci $W^\tau(X,p_1, q_2)$ may be decomposed into a union of products $W^{\alpha}(C_1, p_1, q_1) \times W^{\beta}(C_2, p_2, q_2)$, where the union is taken over \emph{reduced} products $\alpha \beta = \tau$. This is proved in \Cref{prop:decompW}.
What this means for our analysis is that transmission permutations are extremely well-suited to inductive arguments in which curves are repeatedly split into two curves joined at a node. They share this feature with limit linear series, and for this reason \emph{elliptic chains} serve as a natural endpoint for degeneration, just as they do in many applications of limit linear series. If nothing else, one aim of this paper is to make the case that the Demazure product is the most natural underlying combinatorial mechanism for understanding why elliptic chains have been so successful in Brill--Noether theory.

\subsection{Results}
We are concerned in this paper with the following genericity condition. In light of the discussion above, this notion implies classical Brill--Noether generality ($k=0$) and Hurwitz--Brill--Noether generality ($k \geq 2$). In both cases, ``generality'' refers only to loci having the expected dimension, and not to smoothness or stronger conditions.

\begin{defn}
Let $k =0$ or $k\geq 2$, and let $(C,p,q)$ be a twice-marked curve. We say that $(C,p,q)$ has \emph{$k$-general transmission} if
\begin{enumerate}
\item Every transmission permutation $\tau^{p,q}_\cL$ on $(C,p,q)$ is in $\ts{k}$; and
\item For all $\tau \in \ts{k}$, any component of $W^\tau(C,p,q)$ has codimension at least $\inv_k(\tau)$.
\end{enumerate}
In particular, if $\inv_k(\tau) > g$, including the possibility that $\inv_k(\tau) = \infty$, then $W^\tau(C,p,q)$ is empty.
\end{defn}

We conjecture
(\Cref{conj:lowerBoundTtau}) that in fact on \emph{any} $(C,p,q)$ with $kp \sim kq$, every component of $W^\tau(C,p,q)$ has codimension \emph{at most} $\inv_k(\tau)$. This is true for the permutations $\gamma^r_{d-g}$ and $\gamma_{\vec{e}}$ mentioned above by degeneracy locus arguments \cite{kempfSchubert, kleimanLaksov1, kleimanLaksov2} and Larson's theory of splitting loci \cite{larsonDegen,larsonHBN}, respectively. If correct, this definition could be changed to say that $W^\tau(C,p,q)$ is equidimensional of codimension $\inv_k(\tau)$ when $\inv_k(\tau) \leq g$. See \Cref{sec:questions} for some discussion and further conjectures.

Our main result is the following. Here, $\cM_{g,2}$ denotes the moduli stack of twice-marked smooth curves, and $\cH_{g,k,2}$ denotes the substack of $(C,p,q)$ such that $kp \sim kq$, or alternatively the space of degree-$k$ covers $\pi: C \to \PP^1$ with two marked points of total ramification.

\begin{thm}
\label{thm:veryGeneral}
A very general twice-marked curve $(C,p,q)$ in $\cM_{g,2}$ has $0$-general transmission, and such curves are Brill--Noether general; for all $k \geq 2$ a very general point in some component of $\cH_{g,k,2}$ has $k$-general transmission, and such curves are Hurwitz--Brill--Noether general.
\end{thm}

The space $\cH_{g,k,2}$ is reducible in general, since it contains $\cH_{g,d,2}$ for all $d \mid k$.
In fact, if $(C,p,q)$ has $k$-general transmission, then the torsion order of $p-q$ is exactly $k$ (Proposition \ref{prop:kgt-torsion}), so the component of $\cH_{g,k,2}$ in question is a component of the locus where the torsion order of $p-q$ is exactly $k$. This locus is irreducible in characteristic $0$ or in characteristic greater than $k$ \cite[Appendix A]{aggarwal26}. So in these cases, we may remove the phrase ``some component of'' above after restricting to the locus where the torsion order is exactly $k$. The classical Hurwitz space $\cH_{g,k}$ was recently shown to be irreducible in \emph{all} characteristics by Christ, He, and Tyomkin \cite{CHTHurwitz}.

The $k=0$ case of \Cref{thm:veryGeneral} follows from \cite[Theorem 1.12]{pflVersality}, which also gives an existence and smoothness statement, and uses somewhat different methods. It is not clear whether the ``versality of flags'' point of view in that paper can be adapted to the $k \geq 2$ situation.

We will consider certain degenerations, namely to chains of twice-marked curves. For our purposes, a \emph{twice-marked chain} is a twice-marked nodal curve $(X,p_1, q_\ell)$ obtained from $\ell$ twice-marked smooth curves $(C_i, p_i, q_i),$ $1 \leq i \leq \ell$, by gluing $q_i$ to $p_{i+1}$ for $1 \leq i < \ell$. In particular, whenever we say ``twice-marked chain,'' we will always assume that the marked points are at opposite ends of the chain. 
We allow $\ell=1$, so that ``twice-marked chains'' include ``twice-marked smooth curves.''
We will extend the definition of transmission permutations from smooth curves to such chains in \Cref{sec:transmission}. With this terminology in place, we will prove:

\begin{enumerate}
\item In a family of twice-marked chains (possibly including smooth curves), the function sending $(C,p,q) \mapsto \dim W^\tau(C,p,q)$ is upper semicontinuous (\Cref{thm:semicontinuityT}).
\item A genus-$1$ curve has $k$-general transmission if and only if the marked points differ by torsion of order exactly $k$ (\Cref{thm:genus1GT}).
\item General transmission is ``chainable:'' A twice-marked chain has $k$-general transmission if and only if every component in the chain has $k$-general transmission (\Cref{thm:kGTChain}).
\end{enumerate}

In particular, these results together imply \Cref{thm:veryGeneral}, by considering elliptic chains with torsion order $k$ on all components. We must say ``very general'' in that theorem since countably many permutations $\tau$ must be considered. 

\section*{Conventions}
We work over an algebraically closed field of any characteristic.
By \emph{curve}, we always mean a connected, proper, nodal curve. 
A \emph{twice-marked curve} is a curve with two distinct marked smooth points, and a \emph{twice-marked chain} is always assumed to have its marked points at opposite ends of the chain.

The symbol $\NN$ denotes the set of \emph{nonnegative integers}. The symbol $\delta$ always denotes an indicator function, equal to $1$ if the statement within is true, and $0$ otherwise. A \emph{permutation} always refers to a permutation of $\ZZ$.

\section{The Demazure product on \texorpdfstring{$\asp$}{ASP}}
\label{sec:demazure}

This section summarizes material about the Demazure product on $\asp$ from \cite{pflDemazure}, much of which is standard in Coxeter groups.

The Demazure product on $\asp$ may be defined via the functions
$$\sa(a,b) = \# \{n \geq b: \alpha(n) < a\}$$
associated to permutations $\alpha \in \asp$. These are called \emph{(submodular) slipface functions} in \cite[Definition 3.1]{pflDemazure}, and are closely related to \emph{rank functions} on finite symmetric groups. The Demazure product is uniquely characterized by ``min-plus matrix multiplication'' \cite[Theorem A]{pflDemazure}
\begin{equation}
\label{eq:demazureDefn}
s_{\alpha \star \beta}(a,b) = \min_{\ell \in \ZZ} \left\{ \sa(a,\ell) + \sbe(\ell,b) \right\}.
\end{equation}

This equation undoubtedly appears strange at first; the reader may wish to ``try it out'' by checking the special case: if $\alpha, \beta^{-1}$ have no inversions in common, then $\alpha \star \beta = \alpha \beta$ (the converse also holds) \cite[Lemma 5.1]{pflDemazure}; this is easiest to understand when $\beta$ is a simple transposition.

We will require a criterion for obtaining a permutation from a slipface function.

\begin{cor}[{\cite[Proposition 7.12]{pflDemazure}}]
\label{cor:alphaExists}
Let $s: \ZZ^2 \to \NN$ be a function. Suppose that 
\begin{enumerate}
\item There exist integers $M, \chi$ such that $a - b \leq -M$ implies $s(a,b) = 0$ and $a-b \geq M$ implies $s(a,b) = \chi + a - b$; and
\item For all $a,b \in \ZZ$, $s(a+1,b) - s(a,b) - s(a+1,b+1) + s(a,b+1) \geq 0$ ($s$ is \emph{submodular}).
\end{enumerate}
Then there exists a unique permutation $\alpha \in \asp$ such that $s(a,b) = s_\alpha(a,b)$. The shift of $\alpha$ is the number $\chi$ mentioned in criterion (1). This permutation has \emph{bounded difference}, meaning that $|\alpha(n) - n|$ is bounded for $n \in \ZZ$.
\end{cor}

The functions $\sa$ also define the \emph{Bruhat order} on $\asp$: $\alpha \leq \beta$ means $\sa(a,b) \leq \sbe(a,b)$ for all $a,b \in \ZZ$. We will almost never use Bruhat order to compare permutations with different shifts. Bruhat order provides a second, perhaps more intuitive definition of the Demazure product \cite[Theorem B]{pflDemazure}: it is the Bruhat-maximum of all ordinary products of Bruhat-smaller permutations:
\begin{equation}
\label{eq:starMax}
\alpha \star \beta = \max \{ \alpha_1 \beta_1: \alpha_1 \leq \alpha, \beta_1 \leq \beta \}.
\end{equation}
The shift of a permutation $\alpha \in \asp$ determines the asymptotic behavior of $\sa$. This is revealed by the following identity \cite[Equation (15)]{pflDemazure}.
\begin{equation}
\label{eq:duality}
\sa(a,b) - \sai(b,a) = \ca + a - b
\end{equation}
In particular, if $\ca = \cb$ then $\alpha \leq \beta$ if and only if $\alpha^{-1} \leq \beta^{-1}$ \cite[Lemma 2.1]{pflDemazure}.

The shift map $\alpha \mapsto \chi_\alpha$ is a homomorphism for both $\star$ and ordinary multiplication \cite[Equation (16) and Theorem 4.4]{pflDemazure}.
\begin{equation}
\label{eq:shiftHom}
\chi_{\alpha \star \beta} = \chi_{\alpha \beta} = \ca + \cb
\end{equation}
A crucial step in our argument is the reduction of Demazure products to ordinary products. The definition below is adapted from \cite[Lemma 6.2]{pflDemazure}. The fact that $\ts{k}$ is closed under $\star$ and the auxiliary operation $\triangleleft$ used in \cite[Theorem 6.5]{pflDemazure} is noted in \cite[\S 8.4]{pflDemazure}. With these in hand, the following theorem is the specialization of \cite[Theorem 6.5]{pflDemazure} to the subgroup $\ts{k}$.

\begin{defn}
A tuple $(\alpha_1, \cdots, \alpha_\ell)$ of elements of $\asp$ is called \emph{reduced} if 
$\Inv(\alpha_1 \cdots \alpha_\ell)$ is equal to the \emph{disjoint} union of the sets
$\left\{ \left( (\alpha_{n+1} \cdots \alpha_\ell)^{-1}(u), \left( \alpha_{n+1} \cdots \alpha_\ell\right)^{-1}(v) \right):\ (u,v) \in \Inv(\alpha_n) \right\}$ for $1 \leq n \leq \ell$.
\end{defn}

\begin{thm}[{\cite[Theorem 6.5]{pflDemazure}}]
\label{thm:reduce}
Let $\alpha_1, \cdots, \alpha_\ell, \gamma \in \ts{k}$, and suppose $\alpha_1 \star \cdots \star \alpha_\ell \geq \gamma$ and $\sum \chi_{\alpha_n} = \chi_\gamma$. Then there exists a \emph{reduced} $\ell$-tuple $(\beta_1, \cdots, \beta_\ell)$ in $\ts{k}$ such that $\chi_{\beta_i} = \chi_{\alpha_i}$ and $\beta_i \leq \alpha_i$ for all $i$, and $\beta_1 \star \cdots \star \beta_\ell = \beta_1 \cdots \beta_\ell = \gamma$. In particular, $\sum_{i=1}^\ell \inv_k(\beta_i) = \inv_k(\gamma)$.
\end{thm}

\section{Transmission permutations on smooth curves and chains}
\label{sec:transmission}

Let $(C,p,q)$ be a twice-marked smooth curve of genus $g$, and $\cL$ be a degree $d$ line bundle on $C$.
We gave a description of the transmission permutation $\tpql$ in the introduction. In this section we revisit that definition, prove that it is well-defined, and then extend it to chains of curves.
First, define the \emph{transmission function} $\spql$ by
$$\spql(a,b) = h^0\Big( C, \cL\big(\left(a-1\right)p-bq\big)\Big).$$ 
We write $a-1$ rather than $a$ in this definition because it is a necessary (though headache-inducing!) correction to ensure that transmission permutations are combined using the Demazure product.
Let $\chi = d-g$. Riemann--Roch implies that $\spql$ satisfies Criterion (1) of \Cref{cor:alphaExists}, and Criterion (2) follows from the observation that, letting $V_{a,b} = H^0\Big(C, \cL\big((a-1)p - bq\big)\Big)$ and regarding all these as subspaces of $H^0(C \backslash \{p,q\}, \cL)$,
$$\spql(a+1,b) - \spql(a,b) - \spql(a+1,b+1) + \spql(a,b+1) = \dim V_{a+1,b} / \left( V_{a,b} + V_{a+1,b+1} \right) \geq 0.$$

Furthermore, Equation~\eqref{eq:duality} and Riemann--Roch imply $\sti(b,a) = h^1\Big(C, \cL\big((a-1)p-bq\big)\Big)$, hence

\begin{deflem}
\label{lem:tauExists}
For any twice-marked smooth curve $(C,p,q)$ and line bundle $\cL$,
there exists a permutation $\tau = \tpql$ satisfying 
$\spql = \st \text{ and } s^{q,p}_{\omega_C(p+q) \otimes \cL^\vee} = \sti$, and therefore
Equations \eqref{eq:tauh0}, \eqref{eq:tauh1}.
Call $\tpql$ the \emph{transmission permutation} of $\cL$ on $(C,p,q)$.
\end{deflem}

As mentioned in the introduction, these definitions imply that
\begin{equation} \label{eq:shiftTau}
    \chi_{\tpql} = \deg \cL - g.
\end{equation}

\begin{rem}
\label{rem:shiftTwist}
If $\iota_n$ denotes the shift permutation $\iota_n(m) = m-n$, then $\chi_{\iota_n} = n$, and 
\[s_{\iota_n \alpha}(a,b) = \sa(a+n,b)\]
 for all $a,b \in \ZZ$. So $\iota_n \tau^{p,q}_\cL = \tau^{p,q}_{\cL(np)}$ for all line bundles $\cL$, and $W^{\tau}(C,p,q) \cong W^{\iota_n \tau}(C,p,q)$ for all permutations $\tau$, via the ``twist by $np$'' map. This is a convenient way to reduce certain statements to the shift $0$, or alternatively to the degree $0$, case.
\end{rem}

\subsection{Chains of twice-marked curves}
We now extend our definition of transmission loci to certain nodal curves, namely twice-marked chains. 
We work with chains for simplicity, but similar definitions can be made for curves of compact type, and the reader familiar with the theory of enriched structures in the sense of Main\`o \cite{mainoEnrichedStable} (see also \cite{bieselHolmesEnriched,abreuPaciniEnriched}) will see that these definitions naturally extend to that context as well. The primary catch is that transmission functions need not be submodular, so transmission permutations may not exist, in these more general settings. This is related to the fact that in the tropical context \cite{pflBriefTropBN}, transmission functions are not submodular in general.

Fix a single twice-marked chain $(X, p_1, q_\ell)$ obtained by joining $(C_1, p_1, q_1), \cdots, (C_\ell, p_\ell, q_\ell)$. For each node $q_i \in \{q_1, \cdots q_{\ell-1}\}$, there is a line bundle (unique up to isomorphism) $\cY_i$
such that 
$$
\cY_i \mid_{C_i} \cong \cO_{C_i}(-q_i),\ \cY_i \mid_{C_{i+1}} \cong \cO_{C_{i+1}}(p_{i+1}) \text{ and } \cY_i \mid_{C_j} \cong \cO_{C_j} \text{ for all } j \neq i,i+1.$$
 For any $\vec{n} \in \ZZ^{\ell-1}$, define $\cY(\vec{n}) = \bigotimes_{i=1}^{\ell-1} \cY_i^{\otimes n_i}$. 

\begin{defn}
\label{defn:tauChain}
The \emph{transmission function} $s^{p_1, q_\ell}_\cL: \ZZ^2 \to \NN$ of a line bundle $\cL$ on $X$ is
$$s^{p_1, q_\ell}_\cL(a,b) = \min_{\vec{n} \in \ZZ^{\ell-1}} h^0(X, \cL((a-1)p_1-bq_\ell) \otimes \cY(\vec{n})).$$
\end{defn}

\begin{defn} \label{def:WtauChain}
Let $\vec{d} = (d_1, \cdots, d_\ell) \in \ZZ^\ell$, let $d = \sum d_i$, and
denote by $\Pic^{\vec{d}}(X)$ the component of the Picard scheme parameterizing line bundles of degree $d_i$ on $C_i$ for each $i$.
Let $\tau \in \asp$. For any choice of $\vec{d}$ satisfying $\chi_\tau = d - g$, define
$$W^\tau_{\vec{d}}(X,p_1, q_\ell) = \Big\{ [\cL] \in \Pic^{\vec{d}}(X):\ s^{p_1,q_\ell}_\cL(a,b) \geq \st(a,b) \text{ for all } a,b \in \ZZ \Big\}.$$
Define the \emph{transmission locus} $W^\tau(X,p_1, q_\ell)$ to be $W^\tau_{(d,0,\cdots,0)}(X, p_1, q_\ell)$.
\end{defn}

\begin{defn} \label{def:kGTChain}
As for a single curve, we say that a twice-marked chain $(X,p_1,q_\ell)$ has \emph{$k$-general transmission} if every transmission permutation $\tau^{p_1, q_\ell}_\cL$ is in $\ts{k}$, and for all $\tau \in \ts{k}$, any component of $W^\tau(X,p_1,q_\ell)$ has codimension at least $\inv_k(\tau)$.
\end{defn}

The choice $\vec{d} = (d,0,\cdots,0)$ matches the notation of \cite{lieblichOsserman}, but any other choice would work as well: the various loci $W^\tau_{\vec{d}}$ differ merely by twists of $\cY(\vec{n})$.
Note, by Equation~\eqref{eq:duality} and Riemann--Roch, this bound on $h^0(X, \cL(ap_1-bq_\ell) \otimes \cY(\vec{n}))$ is equivalent to bounding $h^1(X, \cL(ap_1-bq_\ell) \otimes \cY(\vec{n}))$ by $\sti(b,a+1)$.

To carry out our analysis, we require an alternative form for $s^{p_1, q_\ell}_\cL$ stated purely in terms of the individual line bundles. To obtain it, we require a basic lemma about line bundles on nodal curves. The reader who is frustrated with the insidious ``$-1$''s that creep into many of our definitions (such as that of $s^{p_1,q_\ell}_\cL$ above) may direct their frustration at this lemma, whence these goblins originate.

\begin{lemma}
Let $X$ be a nodal curve, with a node $p$ that separates it into two nodal curves $X_1, X_2$. Let $\cL$ be a line bundle on $X$, with $\cL_1, \cL_2$ the restrictions to $X_1, X_2$. For $n \in \ZZ$, let $\cY(n)$ be a line bundle on $X$ with $\cY(n) \mid_{X_1} \cong \cO_{X_1}(-np)$ and $\cY(n) \mid_{X_2} \cong \cO_{X_2}(np)$.
Then
$$
\min_{n \in \ZZ} \left\{ h^0(X, \cL \otimes \cY(n)) \right\} = \min_{n \in \ZZ} \left\{ h^0(X_1, \cL_1(-np)) + h^0(X_2, \cL_2((n-1)p)) \right\}.
$$
\end{lemma}

\begin{proof}
This follows from the claim: 
$$ h^0(X, \cL) = \min \Big\{ h^0(X_1, \cL_1) + h^0(X_2, \cL_2(-p)), h^0(X_1, \cL_1(-p)) + h^0(X_2, \cL_2) \Big\}.$$
To prove this claim, consider the exact sequence $0 \to \cL \to \cL_1 \oplus \cL_2 \to \cL \mid_p \to 0$ of sheaves on $X$. The last term is of course isomorphic to $\cO_p$. Taking global sections, it follows that 
\[ h^0(X,\cL) = h^0(X_1, \cL_1) + h^0(X_2, \cL_2) - \delta,\]
 where $\delta$ is $0$ if both $\cL_1, \cL_2$ have a base point at $p$, and $1$ otherwise. In other words, 
 \[ \delta = \max \Big\{ h^0(X_1, \cL_1) - h^0(X_1, \cL_1(-p)), h^0(X_2, \cL_2) - h^0(X_2, \cL_2(-p)) \Big\};\]
  the claim follows. The lemma follows by applying the claim to all bundles $\cL \otimes \cY(n)$ and taking the minimum.
\end{proof}

By induction on $\ell$, we can reduce $s^{p_1, q_\ell}_\cL$ to transmission functions of the components as follows.

\begin{cor}
\label{cor:chainS}
Let $\cL$ be a line bundle on the chain $(X,p_1, q_\ell)$, and denote by $\cL_i$ the restriction to $C_i$. For all $n_0,n_{\ell} \in \ZZ$,
\[\ds s^{p_1, q_\ell}_\cL(n_0,n_{\ell}) = \min_{n_1, \cdots, n_{\ell-1} \in \ZZ} \left\{ \sum_{i=1}^\ell s^{p_i, q_i}_{\cL_i} (n_{i-1}, n_i) \right\}.
\]
\end{cor}

In light of \Cref{lem:tauExists} and the existence and definition of the Demazure product (in this case, an iterated Demazure product), this shows that chains don't just have transmission \emph{functions}; they too have transmission \emph{permutations}, and attachment at a node corresponds to the Demazure product.

\begin{cor}
\label{cor:chainT}
For any degree-$d$ line bundle $\cL$ on the chain $(X,p_1, q_\ell)$, there exists a permutation $\tau = \tau^{p_1, q_\ell}_\cL$ of shift $\chi_\tau = d-g$ such that $s^{p_1, q_\ell}_\cL = s_\tau$. This permutation is given by
$$\tau^{p_1, q_\ell}_\cL = \tpql[1] \star \tpql[2] \star \cdots \star \tpql[\ell].$$ 
\end{cor}

\begin{prop}
\label{prop:decompW}
Let $\tau \in \ts{k}$, let $(X,p_1, q_\ell)$ be a twice-marked chain as above. Let $d = \chi_\tau+g$, where $g$ is the genus of $X$, and let $\vec{d} \in \ZZ^{\ell}$ sum to $d$. Let $W$ be the set of \emph{reduced} tuples $(\alpha_1, \cdots, \alpha_\ell)$ in $\ts{k}$ that satisfy $\chi_{\alpha_i} = d_i - g_i$ for all $i$, where $g_i$ is the genus of $C_i$, and such that $\alpha_1 \alpha_2 \cdots \alpha_\ell = \tau$. Identify $\Pic^{\vec{d}}(X)$ with $\Pic^{d_1}(C_1) \times \Pic^{d_2}(C_2) \times \cdots \times \Pic^{d_\ell}(C_\ell)$. Then
$$W^\tau_{\vec{d}}(X,p_1,q_\ell) = \bigcup_W \prod_{i=1}^{\ell} W^{\alpha_i}(C_i, p_i, q_i).$$
\end{prop}

\begin{proof}
A bundle $[\cL] \in \Pic^{\vec{d}}(X)$ lies in $W_{\vec{d}}^\tau(X,p_1,q_\ell)$ if and only if $ \tpql[1] \star \tpql[2] \star \cdots \star \tpql[\ell] \geq \tau$ in Bruhat order. 
By assumption on $\cL$ and \Cref{eq:shiftTau}, the shift $\chi_i$ of $\tpql[i]$ is $\deg \cL|_{C_i} - g_i$, so $\chi_i = d_i - g_i$ for all $i$.
By \Cref{thm:reduce}, this occurs if and only if there exists a reduced product $\alpha_1 \cdots \alpha_\ell = \tau$ with $\chi_{\alpha_i} = \chi_i$ and $\alpha_i \leq \tpql[i]$ for all $i$. These inequalities are equivalent to $[\cL] \in \prod_{i=1}^{\ell} W^{\alpha_i}(C_i, p_i, q_i)$. So $W^\tau_{\vec{d}}(X,p_1,q_\ell)$ is equal to the union of all such products for $\alpha_1, \cdots, \alpha_\ell$ chosen from the set $W$.
\end{proof}

\begin{eg}
\label{eg:reducedWords}
Suppose that $\inv_k(\tau) = g$, and $X$ is a chain of twice-marked genus $1$ curves of torsion order exactly $k$.
Assume also that $\chi_\tau = 0$.
It will follow from the analysis in \Cref{sec:genus1} that every $(\alpha_1, \cdots, \alpha_\ell) \in W$ contributing a nonempty summand has $\inv_k(\alpha_n) = 1$ for all $n$. If we choose $\vec{d} = (1,\cdots,1)$, we have $\chi_i = 0$ for all $i$, and the contributing elements of $W$
form the set of \emph{reduced words} for $\tau$ in the affine symmetric group. Therefore we obtain a bijection between the points of $W^\tau(X,p_1,q_\ell)$ and reduced words. 
\end{eg}

\begin{thm}
\label{thm:kGTChain}
A twice-marked chain of $k$-torsion curves has $k$-general transmission if and only if each curve in the chain has $k$-general transmission. 
\end{thm}

\begin{proof}
Suppose each $(C_i, p_i, q_i)$ has $k$-general transmission. Then for every choice of line bundles $\cL_i$ on $C_i$, $\tpql[i] \in \ts{k}$ for all $i$, so \Cref{cor:chainT} implies that $\tau^{p_1, q_\ell}_\cL \in \ts{k}$ as well, since $\ts{k}$ is closed under $\star$. 
We now consider the dimension bound. For every $\tau \in \ts{k}$, every element of $W$ has $\inv_k(\tau) = \sum_{i=1}^\ell \inv_k(\alpha_i)$, so every element of the union in that proposition has codimension at least $\inv_k(\tau)$. It follows that $(X,p_1, q_\ell)$ has $k$-general transmission. 

Conversely, suppose that the chain $(X,p_1, q_\ell)$ has $k$-general transmission. 
To tame an illegible nest of subscripts in our notation, define $s_n = s_{\tau^{p_n,q_n}_{\cO_{C_n}}}$ for $n=1,\cdots,\ell$.
We claim that each twice-marked curve $(C_n, p_n, q_n)$ has $k$-torsion, or equivalently that $s_n(k+1,k) \geq 1$. To see this, note that $\tau^{p_1, q_\ell}_{\cO_X}$ is the Demazure product of all $\tau^{p_i,q_i}_{\cO_{C_i}}$. Since the chain as a whole has $k$-general transmission, we have $s_1 \star \cdots \star s_\ell(k+1,k) \geq 1$.
It follows that for each $n$,
$$\left[ \sum_{i=1}^{n-1} s_i(k+1,k+1) \right] + s_n(k+1,k) + \left[ \sum_{i=n+1}^\ell s_i(k,k) \right] \geq 1.$$
For all $i$, $s_i(k+1,k+1) = s_i(k,k) = 0$, since these are dimensions of spaces of sections of negative-degree line bundles. 
It follows that $s_n(k+1,k) = h^0(C_n, \cO_{C_n}(kp_n - kq_n)) \geq 1$, which proves the claim. Therefore all transmission permutations on all components $(C_n, p_n, q_n)$ lie in $\ts{k}$. 

We must now verify the codimension bound on each component. Let $\tau$ be any shift-$0$ permutation, and let $1 \leq n \leq \ell$. 
Identify $W^\tau(X,p_1, q_\ell)$ with $W^\tau_{\vec{d}}(X,p_1, q_\ell)$ with $\vec{d} = (g_1, \cdots, g_\ell)$, where $g_i$ is the genus of $C_i$.
Decompose $W^\tau_{\vec{d}}(X,p_1, q_\ell)$ as in \Cref{prop:decompW}. Note that in that decomposition, each shift $\chi_i$ is $d_i - g_i = 0$. Define a tuple $(\alpha_1, \cdots, \alpha_\ell)$ by $\alpha_n = \tau$ and $\alpha_i = \iota_0$ (the identity) for $i \neq n$. These all have shift $0$. This tuple is an element of the set $W$, and $W^{\alpha_i}(C_i,p_i,q_i)$ is all of $\Pic^{g_i}(C_i)$ if $i \neq n$, and $W^\tau(C_n, p_n,q_n)$ when $i=n$. So $\prod_{i=1}^\ell W^{\alpha_i}(C_i,p_i,q_i)$ has the same codimension in $\Pic^{\vec{d}}(X)$ as $W^\tau(C_n, p_n, q_n)$ does in $\Pic^{g_n}(C_n)$, which is therefore at least $\inv_k(\tau)$. By \Cref{rem:shiftTwist}, this implies that the same holds for all $\tau$ of $\emph{any}$ shift. So all $(C_n, p_n, q_n)$ have $k$-general transmission.
\end{proof}

\section{The genus \texorpdfstring{$1$}{1} case}
\label{sec:genus1}

The story of transmission permutations is particularly simple in genus $1$, which makes it an excellent base case. This section proves

\begin{thm}
\label{thm:genus1GT}
A genus-$1$ twice-marked curve $(E,p,q)$ has $k$-general transmission if and only if it has torsion order $k$. 
\end{thm}

Fix the following notation. For $k=0$ or $k \geq 2$ and any integer $m$, let $\skm$ be
the permutation exchanging $n$ and $n+1$ for all $n \equiv m \pmod{k}$, and fixing all other integers. Define $\iota_m$ as in \Cref{rem:shiftTwist}. It will be useful to have a formula for the following function.
\begin{equation}
\label{eq:sis}
s_{\iota_n \sigma^k_m}(a,b) = \max\{a-b+n,0\} + \delta \Big[ a+n = b\; \text{and}\; b \equiv m+1 \pmod{k} \Big].
\end{equation}

We use the following facts: for any $\alpha \in \ts{k}$, $\inv_k(\alpha) = 0$ if and only if $\alpha = \iota_n$ for some $n$, and $\inv_k(\alpha) = 1$ if and only if $\alpha = \iota_n \skm$ for some $n,m$. For $k \ge 2$, these are standard facts for the affine symmetric group $\widetilde{S}_k$; see for example \cite[\S 8.3]{bjornerBrenti}. Indeed, the affine symmetric group $\widetilde{S}_k$, with the generators $\sigma^k_m$, forms a Coxeter system, in which $\inv_k$ is the length function. The generalization to $\ts{k}$ follows since $\widetilde{S}_k$ is the shift-$0$ subgroup of $\ts{k}$, and multiplication by $\iota_n$ preserves $\inv_k$. The case $k=0$ follows from standard facts about symmetric groups.

\begin{lemma}
\label{lem:genus1}
Suppose $(E, p,q)$ is a genus $1$ twice-marked curve with torsion order $k$, and let $\cL$ be a degree $d$ line bundle on $E$.
\begin{enumerate}
\item If there exists $m \in \ZZ$ such that $\cL \cong \cO_E(m q + (d-m) p)$, then $\tau^{p,q}_\cL = \iota_{d-1} \sigma_{m-1}^{k}$.
\item If no such $m$ exists, then $\tau^{p,q}_\cL = \iota_{d-1}$.
\end{enumerate}
\end{lemma}

\begin{proof}
Riemann--Roch implies that the only special line bundle on a genus $1$ curve is $\cO_E$. Therefore if $\cL$ is not isomorphic to $\cO_E(mq + (d-m)p)$ for any integer $m$, then $s^{p,q}_\cL = s_{\iota_{d-1}}$. This proves part (2). 

Now assume $\cL \cong \cO_E(m q + (d-m) p)$ for some $m$. Then $\cL(ap-bq)$ is special if and only if it has degree $0$ and $b \equiv m \pmod{k}$, and it has $h^0 = h^1 = 1$ in that case. Therefore
$$s^{p,q}_\cL(a,b) = \max \{ d -1 + a - b, 0 \} + \delta \left( b \equiv m\pmod{k} \text{ and } a = b-d+1 \right).$$
By Equation~\eqref{eq:sis}, this is $s_{\iota_{d-1} \sigma^k_{m-1}}(a,b)$.
\end{proof}

\begin{proof}[Proof of \Cref{thm:genus1GT}]
If $(E,p,q)$ is genus $1$ and has torsion order $k$, then \Cref{lem:genus1} shows that $W^{\iota_n}(E,p,q) = \Pic^{n+1}(E)$, $W^{\iota_n \skm}(E,p,q)$ is a single point, and $W^\alpha(E,p,q)$ is empty for any other $\alpha \in \ts{k}$. So $(E,p,q)$ has $k$-general transmission.

On the other hand, for any $k' \neq k$, the fact that $\sigma^k_0$ occurs as a transmission permutation on $(E,p,q)$ implies that $(E,p,q)$ does not have $k'$-general transmission: either $k \nmid k'$, in which case $\sigma^k_0 \not\in \ts{k'}$, or $k \mid k'$, in which case $\inv_{k'}(\sigma^k_0) = k' / k \geq 2$ when $k' \ne 0$, or $\inv_{k'}(\sigma^k_0) = \infty$ when $k' = 0$. But $W^{\sigma^k_0}(E,p,q)$ has codimension $1$.
\end{proof}

\section{Relative transmission loci}
\label{sec:relative}

The purpose of this section is to prove the following theorem. 
\begin{thm}
\label{thm:semicontinuityT}
Let $\mathcal{C}_{g,2}$ denote the locus in $\overline{\cM}_{g,2}$ of twice-marked chains, including twice-marked smooth curves. Then $(X,p,q) \mapsto \dim W^\tau(X,p,q)$ is an upper semicontinuous function on $\mathcal{C}_{g,2}$.
\end{thm}

To do so, we formulate a relative version of transmission loci, defined for versal deformations.
For simplicity, we have not attempted to define relative transmission loci for more general families, and instead freely make assumptions that will simplify the exposition and be sufficient for our purposes. We freely use various standard facts about deformation theory of nodal curves with marked points; a nice summary of what is needed, with references, can be found in \cite[p. 20-21]{lieblichOsserman}.

\subsection{Versal deformation of a chain}
Begin with a single twice-marked chain $(X_0, p_0, q_0)$ of genus $g$. We will consider a versal deformation of $(X_0, p_0, q_0)$, i.e.\ a smooth morphism from a base scheme $B$, which we assume smooth and irreducible, to the stack $\overline{\cM}_{g,2}$. 
Note that here, the subscript $0$ refers to a deformation parameter, rather than indexing components in the chain.
This amounts to a flat proper morphism $\pi: X \to B$ with two sections $p,q: B \to X$. We denote the members of this family by $(X_b, p_b, q_b)$ for geometric points $b \in B$. Since the universal curve is a smooth stack, the total space $X$ is smooth. The morphism $\pi$ is not smooth, of course, but after shrinking $B$ if necessary we can identify that non-smooth locus as a disjoint union of $\ell-1$ codimension-$2$ subschemes $Z_1, \cdots Z_{\ell-1}$, corresponding to the nodes of $X_0$. In each member $(X_b, p_b, q_b)$ of the family, the components $Z_i$ meeting $X_b$ are in bijection with the nodes of $X_b$, which can each be viewed as a node of $X_0$ that has not been smoothed as $X_0$ deforms to $X_b$.
 The images $\pi(Z_i)$ are locally principal subschemes in $B$; shrinking $B$ if necessary we assume that they are principal. Each $\pi^{-1}(\pi(Z_i))$ is a principal divisor in $X$, which can be decomposed into a union of two divisors $Y_i, Y_i'$ meeting transversely at $Z_i$; we take $Y_i$ to be the ``upper half,'' so that for each $b \in \pi(Z_i)$, $X_b \cap Y_i$ contains $p_b$, while $Y'_i$ is the ``lower half'' containing $q_b$ for all such $b$. See \Cref{fig:deform}.

\begin{figure}
\centering
\resizebox{0.5\columnwidth}{!}{\input{figure2-curves-code}}
\caption{A versal deformation of a chain, and the divisors $Y_i, Y'_i$.}
\label{fig:deform}
\end{figure}

The divisors $Y_i, Y'_i$ explain why transmission loci were defined the way that they were for chains. Two line bundles $\cL, \cL'$ on the family $X$ agree on the smooth members if one is obtained from the other via twisting by these divisors. So it is reasonable to require our dimension bounds hold for \emph{all} such twists when working in this family. Now, the line bundles $\cO(Y_i)$ restrict to the family members meeting $Z_i$ (i.e.\ where the $i$th node has not been smoothed) as follows.
\begin{equation}
\label{eq:oxy}
\cO_X(Y_i) \mid_{Y_i} \cong \cO_{Y_i}(-Z_i),\text{ and } \cO_X(Y_i) \mid_{Y'_i} \cong \cO_{Y'_i}(Z_i).
\end{equation}
The assumption that $\pi(Z_i)$ is principal means that $\cO_X(Y_i') \cong \cO_X(-Y_i)$.

\subsection{Relative transmission loci} \label{ss:relative}
This suggests that we ought to define a relative transmission locus for this family as follows. 
For an $(\ell-1)$-tuple $\vec{n} \in \ZZ^{\ell-1}$, let $\cY(\vec{n}) = \cO_X\left( \sum_{i=1}^{\ell-1} n_i Y_i\right)$.
Fix a permutation $\tau \in \asp$ and let $d = \chi_\tau + g$. Denote by $\Pic^{d,p}(\pi) \to B$ the relative Picard scheme of line bundles having degree $d$ on the component of each fiber containing $p_b$ and degree $0$ on every other component (as in \cite[Notation 3.2.4]{lieblichOsserman}), and let $\rho: X \times_B \Pic^{d,p}(\pi) \to \Pic^{d,p}(\pi)$ be the projection. Let $\cU$ be a Poincar\'e line bundle on $X \times_B \Pic^{d,p}(\pi)$. The points of $\Pic^{d,p}(\pi)$ are pairs $(b,[\cL])$ of a point $b \in B$ and the isomorphism class of a line bundle $\cL$ on $X_b$ of the prescribed multidegree. Define a subscheme of $\Pic^{d,p}(\pi)$ as follows.
\begin{eqnarray*}
W^\tau(\pi,p,q) = \Big\{ x \in \Pic^{d,p}(\pi):\ &\dim \left(R^1 \rho_\ast \left(\cU\left(ap - bq \right) \otimes \cY(\vec{n}) \right)\right)_x \geq \sti(b,a+1)\\ &\text{ for all } a,b \in \ZZ,\ \vec{n} \in \ZZ^{\ell-1}\Big\}.
\end{eqnarray*}
Here we abuse notation slightly and write $p,q, \cY(\vec{n})$ for the pullbacks of divisors and line bundles on $X$ to $X \times_B \Pic^{d,p}$. For each choice of $a,b \in \ZZ, \vec{n} \in \ZZ^{\ell-1}$, the locus of $x \in \Pic^{d,p}(\pi)$ satisfying the inequality above has the natural structure of a closed subscheme, cut out by a \emph{Fitting ideal} of the sheaf $R^1 \rho_\ast \left(\cU(ap - bq) \otimes \cY(\vec{n})\right)$. Therefore $W^\tau(\pi,p,q)$ is a closed subscheme.

This construction is only useful if its fibers over smooth members $(X_b, p_b, q_b)$ coincide with the construction of $W^\tau(X_b, p_b, q_b)$. Fortunately, they do. Since the fibers of $\rho$ are $1$-dimensional, the theorem on cohomology and base change implies that, for all $t \in B$, 
the fiber $W^\tau(\pi,p,q)_t$ over $t$ is
\begin{equation}
\label{eq:tpiFiber}
\Big\{ [\cL] \in \Pic^{d,p_t}(X_t): h^1(X_t, \cL(ap_t - bq_t) \otimes \cY(\vec{n})_t ) \geq \sti(b,a+1) \text{ for all } a,b \in \ZZ, \vec{n} \in \ZZ^{\ell-1} \Big\},
\end{equation}
and when $X_t$ is smooth, we have $\cY(\vec{n})_t \cong \cO_{X_t}$ for all $\vec{n} \in \ZZ^{\ell-1}$. It follows that if $X_t$ is smooth, then the fiber $W^\tau(\pi,p,q)_t$ is none other than $W^\tau(X_t, p_t, q_t)$, as we would hope.

\begin{rem}
\label{rem:h1}
A word on our use of cohomology and base change may clarify why we use $h^1$ rather than $h^0$ here. The fact we have used is that if $f: X \to Y$ is a morphism, and $\cF$ is a sheaf on $X$, flat over $Y$, \emph{such that $H^2(X_y, \cF_y) = 0$ for all $y \in Y$} (e.g.\ if the fibers are one-dimensional as in our situation), then the natural map $\phi^1(y): (R^1 f_\ast \cF)_y \to H^1(X_y, \cF_y)$ is an isomorphism. See \cite[Theorem 12.11]{hartshorne}, whose notation we mimic in this remark.
\end{rem}

On the other hand, when $X_t$ is not smooth, the line bundles $\cY(\vec{n})_t$ on $X_t$ are determined up to isomorphism by Equation~\eqref{eq:oxy}. Importantly, these bundles are completely determined by the nodal curve $X_t$ itself, not anything about the geometry of the family. In fact, upon restricting Equation~\eqref{eq:oxy} to a single fiber $X_t$, we see that the bundle $\cO_X(Y_i) \mid_{X_t}$ is either one of the line bundles $\cY_j$ as in \Cref{defn:tauChain} if $Z_i$ meets $X_t$, or $\cO_{X_t}$ otherwise, and every node of $X_t$ corresponds to one of the $Z_i$. 

\begin{cor}
For all $t \in B$, the fiber $W^\tau(\pi,p,q)_t$ is isomorphic to $W^\tau(X_t, p_t, q_t)$.
\end{cor}

\begin{proof}[Proof of \Cref{thm:semicontinuityT}]
The statement is local, so we may fix a chain $(X_0,p_1,q_\ell)$ and consider a versal deformation $\pi: X \to B$ of $(X,p_1,q_\ell)$ as above. It suffices to verify that $b \mapsto \dim W^\tau(X_b,p_b,q_b)$ is upper semicontinuous on $B$. This is the same as the function $b \mapsto \dim W^\tau(\pi,p,q)_b$, so the theorem follows from semicontinuity of fiber dimension.
\end{proof}

\section{Brill--Noether varieties and splitting loci}
\label{sec:specialPerms}

We now demonstrate that classical Brill--Noether loci and Hurwitz--Brill--Noether splitting loci are special cases of transmission loci, with the same expected codimensions. To do so, we first make a simplification to our description of transmission loci.

\subsection{The essential set}
In principle, our definition of transmission loci involves infinitely many inequalities. However, most of them are redundant. 

\begin{defn}[{\cite[Definition 7.6]{pflDemazure}}]
For any $\alpha \in \asp$, the \emph{essential set} of $\alpha$ is
$$\ess(\alpha) = \{ (a,b) \in \ZZ^2: \alpha^{-1}(a-1) \geq b > \alpha^{-1}(a),\ \alpha^{}(b-1) \geq a > \alpha^{}(b) \}$$
\end{defn}

This definition mirrors the ``essential set'' defined in \cite{fultonSchubert} for degeneracy loci in a finite-dimensional vector space, and plays the same role in our analysis. Its key property is the following.

\begin{prop}[{\cite[Corollary 7.9]{pflDemazure}}]
    Say that a permutation $\alpha$ has \emph{bounded difference} if $\{|\alpha(n)-n| \colon n \in \ZZ\}$ is bounded. If $\alpha,\beta \in \asp$ have the same shift and $\alpha$ has bounded difference, then $\alpha \leq \beta$ if and only if $\sa(a,b) \leq \sbe(a,b)$ for all $(a,b) \in \ess(\alpha)$.
\end{prop}

\begin{cor}
\label{cor:essT}
For any twice-marked chain $(X,p,q)$ of genus $g$, permutation $\tau$ of bounded difference, and choice of $\vec{d}$ with sum $d = \chi_\tau + g$, we have
$$W_{\vec{d}}^\tau(X,p,q) = \Big\{ [\cL] \in \Pic^{\vec{d}}(X): s^{p,q}_\cL(a,b) \geq \st(a,b) \text{ for all } (a,b) \in \ess(\tau) \Big\}.$$
\end{cor}

That is, we need only bound the transmission function for values of $(a,b)$ in the essential set.

\begin{eg} \label{eg:theta}
The permutation $\sigma^k_m$ has essential set $\{ (m'+1, m'+1): m' \equiv m \pmod k\}$, and $s_{\sigma^k_m}(m'+1,m'+1) = 1$ for $m' \equiv m \pmod k$. So on a smooth twice-marked curve $(C,p,q)$ of genus $g \ge 1$, $W^{\sigma^k_m}(C,p,q)$ consists of those degree $g$ line bundles for which $\cL(m' p - m'q - q)$ has a nonzero global section for all $m' \equiv m \pmod k$. Assuming that $k p \sim k q$, this is equivalent to $\cL(mp - mq - q)$ having a nonzero global section, and therefore $W^{\sigma^k_m}(C,p,q)$ is isomorphic to the theta divisor in $\Pic^{g-1}(C)$. In particular, this transmission locus always has codimension $1$, provided that $kp \sim kq$.
\end{eg}

The example above gives a quick proof of the following fact, mentioned in the introduction.

\begin{prop} \label{prop:kgt-torsion}
If a twice-marked smooth curve $(C,p,q)$ of genus $g \ge 1$ has $k$-general transmission, then the torsion order of $p-q$ is exactly $k$.
\end{prop}

\begin{proof}
First, assume $k \ge 2$. Let $k'$ be the exact torsion order of $p-q$, and let $\tau = \sigma^{k'}_0$. By \Cref{eg:theta}, $W^\tau(C,p,q)$ has codimension exactly $1$. We know that $k' \mid k$ since $k$-general transmission implies $k p \sim k q$ (see \Cref{ssec:transmission}), so $\tau \in \ts{k}$, and it has $\inv_k(\tau) = k/k'$. By $k$-general transmission, $k/k' \le 1$, so $k = k'$. In the case $k=0$, we have $\inv(\tau) = \infty$ unless $k' = 0$ as well.
\end{proof}

\subsection{Brill--Noether loci}
For positive integers $g,r,d$ and a genus $g$ curve $C$, the classical Brill--Noether locus $W^r_d(C)$ is defined to be $\{ [\cL] \in \Pic^d(C): h^0(C,\cL) \geq r+1\}$. We are interested only in cases where $r+1, g-d+r \geq 1$, since otherwise $W^r_d(C) = \Pic^d(C)$. In these cases, the expected codimension of $W^r_d(C)$, as predicted, e.g., by Porteous's formula, is $(r+1)(g-d+r)$. \Cref{cor:essT} provides a route to identify Brill--Noether loci with transmission loci; we need only specify the right permutation.

\begin{figure}
\begin{center}
\begin{tikzpicture}[scale=0.2]
\def\a{0}
\def\b{0}
\fill[lightgray] (-0.5+\b, \a+0.5) rectangle (7,-6);
\fill[lightgray] (-0.5+\b, \a+0.5) rectangle (-7,6);

\draw[->] (-7,0) -- (7,0) node[right] {$n$};
\draw[->] (0,-6) -- (0,6) node[above] {$\gamma^r_\chi(n)$};

\foreach \m in {-5,...,-3} {
\node [fill=black, circle, inner sep=1pt] at (\m,-1+\m) {};
}
\foreach \m in {-2,...,-1} {
\node [fill=black, circle, inner sep=1pt] at (\m,3+\m) {};
}
\foreach \m in {0,...,3} {
\node [fill=black, circle, inner sep=1pt] at (\m,-3+\m) {};
}
\foreach \m in {4,...,7} {
\node [fill=black, circle, inner sep=1pt] at (\m,-1+\m) {};
}

\draw[decorate, decoration={brace, amplitude=4pt}]
  (3.5,-0.5) -- (0.5,-3.5)
  node[midway, below right=1pt] {$r+1$};

\draw[decorate, decoration={brace, amplitude=4pt}]
  (-3,0.5) -- (-0.5,3)
  node[midway, above left=1pt] {$r-\chi$};

\end{tikzpicture}
\end{center}
\caption{%
\label{fig:gamma}
An illustration of the permutations $\gamma^r_\chi$. The specific example shown is $\gamma^3_1$. Counting points in the shaded quadrants gives the imposed bounds on $h^0(C,\cL)$ and $h^1(C,\cL)$, which are $r+1$ and $r-\chi$, respectively.
}
\end{figure}

\begin{defn}
Suppose $r \geq \max \{0, \chi+1\}$. Let $\gamma^r_\chi$ be the unique permutation that restricts to the unique \emph{increasing} bijection between the following sets. For readability, we use interval notation, but in each case we mean the intersection with $\ZZ$.
\begin{eqnarray*}
\left(-\infty, -1\right] &\xrightarrow{\sim}& \left(-\infty,-r-1\right] \cup \left[1, r-\chi\right] \\
\left[0, \infty\right) &\xrightarrow{\sim}& \left[ -r,0\right] \cup \left[r-\chi+1,\infty\right)
\end{eqnarray*}
\end{defn}

The following facts about $\gamma^r_\chi$ are straightforward to verify from definitions. 
\begin{lemma}
The essential set of $\gamma = \gamma^r_\chi$ is $\{(1,0)\}$, and $s_\gamma(1,0) = r+1$. The shift of $\gamma$ is $\chi$. The set of inversions of $\gamma$ is $([-(r-\chi), -1] \times [0,r]) \cap \ZZ^2$.
\end{lemma}
\begin{cor}
If $(C,p,q)$ is a genus $g$ smooth twice-marked curve, and $r,d$ are integers with $r \geq 0$ and $g-d+r > 0$, then $W^r_d(C) = W^{\gamma^r_{d-g}}(C,p,q)$, and $\inv(\gamma^r_{d-g}) = (r+1)(g-d+r)$. In particular, if $(C,p,q)$ has $0$-general transmission, then it is Brill--Noether general, in the sense that all $W^r_d(C)$ have dimension exactly $g - (r+1)(g-d+r)$.
\end{cor}

\subsection{Hurwitz--Brill--Noether splitting loci}

Throughout this section, fix an integer $k \geq 2$. As in \cite{larsonLarsonVogt}, let $\cH_{g,k}$ denote the Hurwitz space, parameterizing degree-$k$ covers $\pi: C \to \PP^1$ from a genus $g$ smooth curve, and let $\cH_{g,k,2}$ denote the moduli space of degree-$k$ covers together with two marked points $p,q \in C$ of total ramification. Equivalently, this is the moduli space of twice-marked smooth curves $(C,p,q)$ such that $kp \sim kq$. Hurwitz--Brill--Noether theory concerns the Brill--Noether varieties of a general curve in $\cH_{g,k}$, and early work concerned the determination of $\dim W^r_d(C)$ for such curves \cite{cm99, cm02, pflKGonal, jensenRanganathan}. Cook-Powell--Jensen \cite{cpjComponents, cpjMethods} and Larson \cite{larsonHBN} independently refined the theory by observing that for a curve in $\cH_{g,k}$, $\Pic(C)$ has a much more refined and well-behaved stratification into \emph{splitting loci}. In particular, this refinement is naturally studied by Larson's results on splitting loci \cite{larsonDegen}.
Splitting loci answered a riddle originally present in the formula for $\dim W^r_d(C)$ conjectured in \cite{pflKGonal} and proved in \cite{jensenRanganathan}, which suggested that $W^r_d(C)$ is not equidimensional for general $k$-gonal curves; the reason is that $W^r_d(C)$ decomposes into several ``balanced splitting loci'' that may have different dimensions.
The full suite of classical theorems in Brill--Noether theory has recently been generalized to the context of splitting loci by Larson, Larson, and Vogt \cite{larsonLarsonVogt}; see also \cite{k3hbn} for a second development of this theory via K3 surfaces.

\begin{defn}
A \emph{splitting type} is a nondecreasing $k$-tuple $\vec{e} = (e_1, \cdots, e_k) \in \ZZ^k$. 
For a splitting type $\vec{e}$, let
$d(\vec{e}) 
= g - 1 + \sum_{n=1}^k (e_n+1),$
and let $x_{\vec{e}}: \ZZ \to \NN$ be the function
$$x_{\vec{e}}(m) = \sum_{n=1}^k \max \{ e_n + 1 + m, 0 \}.$$
For a cover $\pi: C \to \PP^1$ in $\cH_{g,k}$ with $P = \pi^\ast \cO_{\PP^1}(1)$, define the \emph{splitting locus}
$$W^{\vec{e}}(C,P) = \{ [\cL] \in \Pic^{d(\vec{e})}(C):\ h^0(C, \cL(mP)) \geq x_{\vec{e}}(m) \text{ for all } m \in \ZZ \}.$$
The \emph{expected codimension} of $W^{\vec{e}}(C,P)$ is the number $u(\vec{e}) = \sum_{1 \leq m,n \leq k} \max \{0, e_m - e_n-1\}$.
\end{defn}

Our notation differs slightly from \cite{larsonLarsonVogt}, in that we specify the divisor class $P$. If $k$ is the gonality of $C$, this can be left implicit, but we include it for emphasis and because for larger $k$ the choice of $P$ need not be unique.
The definition of splitting loci above is reminiscent of our definition of transmission loci, and indeed, this is not a coincidence: we show in this section that for $(C,p,q)$ in $\cH_{g,k,2}$, i.e.\ in the situation where $P \sim kp \sim kq$, splitting loci \emph{are} transmission loci, and the stratification by transmission loci may be viewed as a further refinement of the splitting type stratification. There is another, more geometrically meaningful description of splitting loci that explains their name: $[\cL]$ belongs to the open part of $W^{\vec{e}}(C,P)$ (the complement of all $W^{\vec{f}}(C,P) \subsetneq W^{\vec{e}}(C,P)$) if and only if $\pi_\ast \cL$ is isomorphic to $\cO_{\PP^1}(e_1) \oplus \cdots \oplus \cO_{\PP^1}(e_k)$.

Larson proved in \cite{larsonHBN}, via intersection theory techniques developed in \cite{larsonDegen}, that for \emph{every} point in $\cH_{g,k}$, $W^{\vec{e}}(C,P)$ is nonempty if $u(\vec{e}) \leq g$, and every component has codimension at most $u(\vec{e})$ if so. For a \emph{general point} in $\cH_{g,k}$, Larson \cite{larsonHBN} and Cook-Powell--Jensen \cite{cpjComponents} independently proved that every component of $W^{\vec{e}}(C,P)$ has codimension \emph{at least} $u(\vec{e})$; the results of this paper provide a new proof of that in the broader context of transmission loci. Much stronger results about irreducibility, smoothness, and monodromy of splitting loci can be found in \cite{larsonLarsonVogt}, and we conjecture that the same results should hold for transmission loci (\Cref{conj:llvResults}).

\subsection{Associating permutations to splitting loci}

For a point $(C,p,q)$ of $\cH_{g,k,2}$, we have $kp \sim kq$, and therefore every transmission permutation is in $\ts{k}$. There is a cover $\pi: C \to \PP^1$, totally ramified at $p$ and $q$, so the class of the fiber is $P = kp$. In particular, for all $m \in \ZZ$, $h^0(C, \cL(mP)) = s^{p,q}_\cL(1+mk, 0) = s^{p,q}_\cL(1+ak,bk)$ for all $a,b \in \ZZ$ such that $a-b = m$. So the transmission function, and hence the transmission permutation, of $\cL$ determines its splitting type. 
We demonstrate in this section how to read the splitting type from a permutation in $\ts{k}$, and how to identify a splitting locus with a transmission locus. The content of this section is not novel, and indeed the affine symmetric groups are used systematically to study splitting loci in \cite{larsonLarsonVogt} (see especially Theorem 1.4); we include this section only for completeness and to explain the story in our notation.

Call an extended affine permutation $\alpha \in \ts{k}$ \emph{(affine) bigrassmannian} if $\alpha$ is increasing on $\{0, \cdots, k-1\}$ and $\alpha^{-1}$ is increasing on $\{1, \cdots, k\}$. Bigrassmannian permutations have the property that $(a,b) \in \ess (\alpha)$ implies $a \equiv 1 \pmod{k}$ and $b \equiv 0 \pmod{k}$. In other words, the definition of $W^\alpha(C,p,q)$ imposes conditions only on $\cL$ and twists by multiples of $kp, kq$. Therefore these are the permutations that correspond to splitting types; our object in this section is to describe the correspondence between splitting types and bigrassmannian permutations in $\ts{k}$.

Call a permutation $\rho \in \ts{k}$ \emph{residual} if it restricts to a permutation of $\{ 0, \dots, k-1\}$. Residual permutations constitute a copy of $S_k \le \ts{k}$, given by extending $\rho \in S_k$ by $\rho(n+k) = \rho(n)+k$. For any $\alpha \in \ts{k}$, there exists a unique pair $(\rho, \pi)$ of a residual permutation and a $k$-\emph{periodic} function such that
$$\alpha(n) = \rho(n) + 1 + k \pi(n).$$
Conversely, every such pair $(\rho,\pi)$ gives an $\alpha \in \ts{k}$. This decomposition is convenient for studying the splitting type; the $+1$ above is a convenience, as we will see. If $\alpha$ has this decomposition, then its inverse is
$$\alpha^{-1}(n) = \rho^{-1}(n-1) - k \pi(\rho^{-1}(n-1)).$$
The residual permutation is irrelevant to the values of the slipface determining the splitting type of $\alpha$, because
\begin{eqnarray*}
s_{\alpha}(1+ak,bk) &=& \sum_{n=0}^{k-1} \#\{q \in \ZZ: n + qk \geq bk \text{ and } \alpha(n + qk) < 1 + ak \}\\
&=& \sum_{n=0}^{k-1} \max \{0, a-b-\pi(n)\}.
\end{eqnarray*}
So the splitting type of $\alpha$ is given by sorting the tuple $(-\pi(0)-1, \cdots, -\pi(k-1)-1)$ to nondecreasing order. In particular, this shows that every splitting type occurs for some $\alpha \in \ts{k}$. We may also obtain a useful bound on $\inv_k(\alpha)$ from $\pi$ alone. Observe that
$$
\inv_k(\alpha) \ge \# \left\{ (m,n): 0 \leq n < k,\ \floor{\frac{m}{k}} < \floor{\frac{n}{k}},\ \floor{ \frac{\alpha(m)-1}{k}} > \floor{ \frac{\alpha(n)-1}{k} } \right\},
$$
and equality holds if and only if $\alpha$ is bigrassmannian. This lower bound may be computed from $\pi$ alone. For fixed $0 \leq m,n < k$, consider which pairs $(m-qk, n)$ are counted by this lower bound. We have $\floor{ \frac{m-qk}{k} } < \floor{ \frac{n}{k}}$ if and only if $0 < q$, and $\floor{\frac{\alpha(m-qk)-1}{k}} > \floor{\frac{\alpha(n)-1}{k}}$ if and only if $\pi(m)-q > \pi(n)$. So the number of such pairs is $\max\{0, \pi(m) - \pi(n)-1\}$, and we deduce that 
\begin{equation}
\label{eq:invkPi}
\inv_k(\alpha) \ge \sum_{0 \leq m,n < k} \max\{0, \pi(m) - \pi(n)-1\},
\end{equation}
with equality if and only if $\alpha$ is bigrassmannian.

Conveniently, this equality case coincides with another useful situation: when the essential set consists only of pairs $(1+ak,bk)$. In this case, $W^\alpha(C,p,q)$ is identical to a splitting locus. Note that we use here the fact that if $\alpha$ is increasing on a set, then it is automatically increasing on any translate of that set by a multiple of $k$.
We will now classify the choices of $\rho, \pi$ for which this situation occurs. Observe that if $0 \leq m,n < k$, and we wish to determine which of $\alpha(m), \alpha(n)$ is larger, we can do so by comparing $\pi(m), \pi(n)$, and breaking a tie with $\rho(m), \rho(n)$. From this and the discussion above, we obtain the following classification. 

\begin{lemma}
Say that $\rho$ is \emph{increasing when $\pi$ is tied} if for all $0 \leq m,n < k$, if $\pi(m) = \pi(n)$ then $m<n$ if and only if $\rho(m) < \rho(n)$.
The permutation $\alpha$ defined above is increasing on $\{0, \cdots, k-1\}$ if and only if $\pi(0) \leq \cdots \leq \pi(k-1)$ and $\rho$ is increasing when $\pi$ is tied. On the other hand, $\alpha^{-1}$ is increasing on $\{1,\cdots, k\}$ if and only if $\pi(\rho^{-1}(0)) \geq \pi(\rho^{-1}(1)) \geq \cdots \geq \pi(\rho^{-1}(k-1))$ and $\rho$ is increasing when $\pi$ is tied.
\end{lemma}

\begin{cor}
\label{cor:rho}
For a fixed $k$-periodic function $\pi$, there exists a residual permutation $\rho$ such that $\alpha = \rho + 1 + k \pi$ is bigrassmannian if and only if $\pi(0) \le \cdots \le \pi(k-1)$. If such $\rho$ exists, it is unique.
\end{cor}

\begin{proof}
The lemmas above show that it is necessary for $\pi$ to be nondecreasing on $\{0, \cdots, k-1\}$. If so, then $\alpha$ has the desired property if and only if $\rho$ is increasing when $\pi$ is tied, and precomposing with $\rho^{-1}$ reverses the order of $\pi(0), \cdots, \pi(k-1)$. This uniquely determines $\rho$: $\rho\nv(0), \cdots, \rho\nv(k-1)$ must consist of the indices for which $\pi(n)$ is maximum, in increasing order, followed by the indices where $\pi(n)$ is the second-largest value, in increasing order, and so on. Explicitly,
$$\rho(n) = \#\{ m: 0 \leq m < k,\ \pi(m) > \pi(n) \} + \#\{m: 0 \leq m < n,\ \pi(m) = \pi(n) \}$$ for all $0 \leq n < k.$ Note that this formula closely resembles the construction of the permutation $w(\vec{e})$ in \cite[Theorem 1.4]{larsonLarsonVogt}; up to some conventions, the two constructions are the same.
\end{proof}

\begin{defn}
For any splitting type (nondecreasing $k$-tuple) $\vec{e} = (e_1, \cdots, e_k)$, let $\gamma_{\vec{e}} = \rho + 1 + k \pi$, where $\pi$ is the $k$-periodic function determined by $\pi(n) = - e_{k-n}-1$ for $0 \leq n < k$ (so that $\pi(0) \leq \cdots \leq \pi(k-1)$) and $\rho$ is the permutation described in \Cref{cor:rho}.
\end{defn}

Figure \ref{fig:gammaSplitting} illustrates two examples of $\gamma_{\vec{e}}$.

\begin{figure}
\begin{center}
\begin{tikzpicture}[scale=0.2]
\def\k{2}
\def\xmin{-6}
\def\xmax{6}
\def\ymin{-6}
\def\ymax{6}

\foreach \m in {-3,...,4} {
    \draw[dotted] (2*\m-0.5,\ymin-1.5) -- (2*\m-0.5,\ymax+0.5);
    \draw[dotted] (\xmin-0.5, 2*\m-1.5) -- (\xmax+1.5, 2*\m-1.5);
}

\draw[->] (\xmin-1,0) -- (\xmax+2,0) node[right] {$n$};
\draw[->] (0,\ymin-2) -- (0,\ymax+1.5) node[above] {$\gamma_{(-1,1)}(n)$};
\foreach \m in {\xmin,...,\xmax} {
  \foreach \r/\v in {0/-2,1/1} {
    \pgfmathtruncatemacro{\residue}{mod(\m-\r,\k)}
    \ifnum\residue=0
      \pgfmathtruncatemacro{\gamman}{\v+\m-\r}
      \ifnum\gamman<\ymin\else
        \ifnum\gamman>\ymax\else
          \node[fill=black, circle, inner sep=1pt] at (\m,\gamman) {};
        \fi
      \fi
    \fi
  }
}
\end{tikzpicture}
\qquad
\begin{tikzpicture}[scale=0.15]
\def\k{4}
\def\xmin{-8}
\def\xmax{11}
\def\ymin{-11}
\def\ymax{8}

\foreach \m in {-2,...,3} {
    \draw[dotted] (4*\m-0.5, \ymin-0.5) -- (4*\m-0.5, \ymax+0.5);
    \draw[dotted] (\xmin-0.5, 4*\m-3.5) -- (\xmax+0.5, 4*\m-3.5);
}

\draw[->] (\xmin-1,0) -- (\xmax+1,0) node[right] {$n$};
\draw[->] (0,\ymin-1) -- (0,\ymax+1) node[above] {$\gamma_{(-1,-1,0,1)}(n)$};
\foreach \m in {\xmin,...,\xmax} {
  \foreach \r/\v in {0/-4,1/-1,2/1,3/2} {
    \pgfmathtruncatemacro{\residue}{mod(\m-\r,\k)}
    \ifnum\residue=0
      \pgfmathtruncatemacro{\gamman}{\v+\m-\r}
      \ifnum\gamman<\ymin\else
        \ifnum\gamman>\ymax\else
          \node[fill=black, circle, inner sep=1pt] at (\m,\gamman) {};
        \fi
      \fi
    \fi
  }
}
\end{tikzpicture}
\end{center}
\caption{%
\label{fig:gammaSplitting}
The permutations $\gamma_{(-1,1)}$ and $\gamma_{(-1,-1,0,1)}$.
Their values on $\{0,\ldots,k-1\}$ are $(-2,1)$ and $(-4,-1,1,2)$, respectively; all other values are determined by $\gamma(n+k)=\gamma(n)+k$. 
The dotted lines split the plane into a coarser grid $k \times k$ squares. That $\gamma_{\vec{e}}$ is bigrassmannian means that it is increasing when restricted to a row or column of this coarser grid. That $\gamma_{\vec{e}}$ is $k$-affine means that the graph is preserved when the coarser grid is moved one step right and one step up.
}
\end{figure}

\begin{prop}
The permutation $\gamma = \gamma_{\vec{e}}$ has $s_\gamma(1+ak,bk) = x_{\vec{e}}(a-b)$ for all $a,b \in \ZZ$, $\ess(\gamma) \subseteq \{ (1+ak,bk):\ a,b \in \ZZ\}$, and $\inv_k(\gamma) = u(\vec{e})$.
\end{prop}

\begin{proof}
The computation of $s_\gamma(1+ak,bk)$ is carried out above. The description of $\gamma$ in \Cref{cor:rho} and identity $\gamma(n+k) = \gamma(n) + k$ imply that $\gamma(b) < \gamma(b-1)$ is only possible for $b \equiv 0 \pmod{k}$ and $\gamma^{-1}(a) < \gamma^{-1}(a-1)$ is only possible for $a \equiv 1 \pmod{k}$, hence $(a,b) \in \ess(\gamma)$ implies that $a \equiv 1 \pmod{k}$ and $b \equiv 0 \pmod{k}$.
Equation~\eqref{eq:invkPi} implies that $\inv_k(\gamma) = u(\vec{e})$.
\end{proof}

\begin{cor}
For any point $(C,p,q)$ of $\cH_{g,k,2}$, the splitting locus $W^{\vec{e}}(C,kp)$ is equal to the transmission locus $W^{\gamma_{\vec{e}}}(C,p,q)$, and the expected codimensions match. In particular, if $(C,p,q)$ has $k$-general transmission, then it is \emph{Hurwitz--Brill--Noether general}, in the sense that all its splitting loci have the expected dimension.
\end{cor}

\section{Questions and conjectures}
\label{sec:questions}

We end with some conjectures and questions for future work. We mentioned the first conjecture in the introduction.

\begin{conj}
\label{conj:lowerBoundTtau}
Let $k\geq 2$. On \emph{any} $(C,p,q)$ with $kp \sim kq$, for each $\tau \in \ts{k}$, every component of $W^\tau(C,p,q)$ has codimension \emph{at most} $\inv_k(\tau)$.
\end{conj}

After it was circulated in an early draft of this paper, \Cref{conj:lowerBoundTtau} was proved by Daksh Aggarwal \cite{aggarwal26}, along with other structural results on transmission loci. 

\begin{conj}
\label{conj:llvResults}
All statements of \cite[Theorem 1.2]{larsonLarsonVogt} generalize from splitting loci to transmission loci of a general point in the locus in $\cH_{g,k,2}$ where the torsion order is exactly $k$ (or $\cM_{g,2}$, for $k=0$). In particular, the intersection theory class of $W^\tau(C,p,q)$ is 
$$[W^\tau(C,p,q)] = \frac{N(\tau)}{\inv_k(\tau)!} \Theta^{\inv_k(\tau)},$$
where $N(\tau)$ is the number of reduced words for $\iota_{-\chi_\tau} \tau$ (this $\iota_{-\chi_\tau}$ serves to convert $\tau$ to something of shift $0$). Furthermore, at \emph{any} point of $\cH_{g,k,2}$, $W^\tau(C,p,q)$ supports this intersection class.
\end{conj}

In the case $k=0$, this enumerative formula follows from the results of \cite{pflVersality}, so the $k \geq 2$ case is of primary interest. It seems very plausible that the methods of \cite{larsonLarsonVogt} can be adapted to prove this conjecture. A notable special case is $\inv_k(\tau) = g$; as observed in \Cref{eg:reducedWords} the Demazure product machinery gives a bijection between reduced words for $\tau$ and points of $W^\tau(X,p,q)$ on an elliptic chain $X$, and more generally \Cref{prop:decompW} points to the very explicit link between reduced words and transmission loci on elliptic chains that is provided by the Demazure product.

\begin{conj}
Let $\pi: X \to B$ be a versal family in $\cM_{g,2}$ (if $k=0$) or $\cH_{g,k,2}$ (if $k \geq 2$).
In a relative transmission locus $W^\tau(\pi,p,q)$, every component has codimension at most $\inv_k(\tau)$.
\end{conj}

The importance of this conjecture is that it would allow a ``regeneration theorem'', akin to the regeneration theorem for limit linear series, to be proved for transmission loci: if a transmission locus has the expected dimension on a singular curve, then this conjecture would show that the locus is part of a component that also lies over nearby smooth curves with the expected dimension.

\begin{qu}
For which permutations $\tau$ and genera $g$ does there exist a smooth twice-marked curve $(C,p,q)$ of genus $g$ and line bundle $\cL$ with $\tpql = \tau$? 
\end{qu}

This question is likely to be quite hard to answer in full generality, since it strictly generalizes the question of which Weierstrass semigroups occur on marked algebraic curves, which is still wide open (see, e.g., \cite{kaplanYe}). The next question is a variation that adds an expected dimension requirement, similar to the study of dimensionally proper Weierstrass points in \cite{eh87}, for example.
It is analogous to questions about Brill--Noether theory in negative expected dimension, such as those addressed in \cite{pflThriftyLego} and \cite{teixidorBNLoci}.

\begin{qu}
Given $\tau \in \ts{k}$ and a genus $g < \inv_k(\tau)$, suppose that $(C,p,q)$ and $\cL$ are such that $W^{\tau}(C,p,q)$ is nonempty. Call a point of $W^\tau(C,p,q)$ \emph{dimensionally proper} if, in a versal deformation $\pi: X \to B$ of $(C,p,q)$ in $\cM_{g,2}$ (if $k=0$) or $\cH_{g,k,2}$ (if $k \geq 2$), $(C,p,q,[\cL])$ belongs to a component of $W^\tau(\pi,p,q)$ of dimension exactly the expected dimension $\dim B + g - \inv_k(\tau)$. For which $\tau, g$, do there exist such dimensionally proper points? 
\end{qu}

Finally, I believe that the formula for the algebraic Euler characteristic of Brill--Noether varieties in terms of set-valued Young tableaux \cite{cpEuler,actKclasses} generalizes to transmission loci in the following form. This would in particular give a formula for the Euler characteristics of splitting loci on general $k$-gonal curves. My basis for believing this conjecture is the observation that the combinatorial analysis of \cite{cpEuler} can be translated to the language of transmission permutations. To prove it, however, requires better information about relative transmission loci; most importantly, flatness criteria are required. It also seems likely that the degeneracy locus techniques of \cite{actKclasses,actMotivic} are adaptable to transmission loci.

\begin{conj}
Let $\tau \in \ts{k}$ have shift $0$ and $\inv_k(\tau) \le g$. Let $(C,p,q)$ be a general point in $\cM_{g,2}$ (if $k = 0$) or a general point of the locus in $\cH_{g,k,2}$ where the torsion order is exactly $k$ (if $k \geq 2$). Define a \emph{length-$g$ Hecke word} for $\tau$ to be a factorization
\[ \tau = \alpha_1 \star \alpha_2 \star \cdots \star \alpha_g \]
where each $\alpha_i$ has shift $0$ and $\inv_k(\alpha_i) = 1$. Equivalently, for each $\alpha_i = \sigma^k_m$ for some $m \in \ZZ$. Let $H^\tau_g$ denote the set of length-$g$ Hecke words for $\tau$. Then
\[ \chi \left( W^\tau(C,p,q), \cO_{W^\tau(C,p,q)} \right) = (-1)^{g - \inv_k(\tau)} \; \# H^\tau_g.\]
Note in particular that this formula generalizes \Cref{eg:reducedWords}.
\end{conj}
 


\section*{Acknowledgements}
This work was supported by a Miner D. Crary Sabbatical Fellowship from Amherst College.
Sam Payne, Dave Jensen, and Daksh Aggarwal provided helpful comments on early drafts.
I am also grateful to my son Anders, sixteen months old at the time, who was my principal collaborator during the early days of the pandemic, supervised almost all the work in the first draft of this paper, and exhorted me to rethink my approach to \Cref{sec:specialPerms} by pulling a sheet of notes about it off my desk and attempting to eat it.


\bibliographystyle{amsalpha}
\bibliography{hbnDemazure}
\end{document}

%% file: figure2-curves-code.tex
\begin{tikzpicture}[
  x=1mm,
  y=-1mm,
  line cap=round,
  line join=round
]


\def\figShiftX{0}
\def\figShiftY{0}

\def\xLeftTopAnchor{20}
\def\yLeftTopAnchor{0}
\def\xLeftNode{20}
\def\yLeftNode{20}
\def\xLeftBottomAnchor{20}
\def\yLeftBottomAnchor{60}

\def\xSpineTop{50}
\def\ySpineTop{-6}
\def\xSpineUpper{50}
\def\ySpineUpper{16}
\def\xSpineLower{50}
\def\ySpineLower{36}
\def\xSpineBottom{50}
\def\ySpineBottom{54}

\def\xRightTopAnchor{80}
\def\yRightTopAnchor{0}
\def\xRightNode{80}
\def\yRightNode{40}
\def\xRightBottomAnchor{80}
\def\yRightBottomAnchor{60}

\def\xBaseLeftAnchor{20}
\def\yBaseLeftAnchor{98}
\def\xBaseCenter{50}
\def\yBaseCenter{94}
\def\xBaseRightAnchor{80}
\def\yBaseRightAnchor{98}

\def\xArrow{50}
\def\yArrowStart{74}
\def\yArrowEnd{82}

\def\hLeftCtrlOneDX{8}
\def\hLeftCtrlOneDY{0}

\def\hRightCtrlTwoDX{-8}
\def\hRightCtrlTwoDY{0}

\def\hFamilyMidEndDX{-9}
\def\hFamilyMidEndDY{9}
\def\hFamilyMidStartDX{9}
\def\hFamilyMidStartDY{9}

\def\leftWallCtrlOneDX{1}
\def\leftWallCtrlOneDY{8}
\def\leftWallCtrlTwoDX{7}
\def\leftWallCtrlTwoDY{-4}

\def\leftWallLowerCtrlOneDX{7}
\def\leftWallLowerCtrlOneDY{4}
\def\leftWallLowerCtrlTwoDX{1}
\def\leftWallLowerCtrlTwoDY{-12}

\def\spineCtrlOneDX{3}
\def\spineCtrlOneDY{10}
\def\spineCtrlTwoDX{3}
\def\spineCtrlTwoDY{-8}

\def\spineMidCtrlOneDX{3}
\def\spineMidCtrlOneDY{8}
\def\spineMidCtrlTwoDX{3}
\def\spineMidCtrlTwoDY{-8}

\def\spineLowerCtrlOneDX{3}
\def\spineLowerCtrlOneDY{8}
\def\spineLowerCtrlTwoDX{3}
\def\spineLowerCtrlTwoDY{-10}

\def\rightWallCtrlOneDX{-1}
\def\rightWallCtrlOneDY{10}
\def\rightWallCtrlTwoDX{-7}
\def\rightWallCtrlTwoDY{-4}

\def\rightWallLowerCtrlOneDX{-7}
\def\rightWallLowerCtrlOneDY{4}
\def\rightWallLowerCtrlTwoDX{-1}
\def\rightWallLowerCtrlTwoDY{-12}

\def\overHorizOuter{2.4}
\def\overHorizInner{0.35}
\def\overMagentaOuter{2.2}
\def\overBaseOuter{2.2}
\def\overBaseInner{0.35}
\def\overWallOuter{0.8}
\def\overWallInner{0.45}
\def\overSpineOuter{0.8}
\def\overSpineInner{0.45}

\def\curveWidth{2}
\def\arrowWidth{2}
\def\dotRadius{1.2}

\def\yOneShadeRadius{28}
\def\yOneLabelShiftX{0}
\def\yOneLabelShiftY{0}

\def\yTwoShadeRadius{40}
\def\yTwoLabelShiftX{0}
\def\yTwoLabelShiftY{0}

\def\yOnePrimeShadeRadius{34}
\def\yOnePrimeLabelShiftX{0}
\def\yOnePrimeLabelShiftY{0}

\def\yTwoPrimeShadeRadius{28}
\def\yTwoPrimeLabelShiftX{0}
\def\yTwoPrimeLabelShiftY{0}

\def\xZeroLabelShiftX{0}
\def\xZeroLabelShiftY{-2}

\def\pLabelShiftX{4}
\def\pLabelShiftY{0}

\def\qLabelShiftX{4}
\def\qLabelShiftY{0}

\def\zOneLabelShiftX{-2}
\def\zOneLabelShiftY{0}

\def\zTwoLabelShiftX{2}
\def\zTwoLabelShiftY{0}

\def\piLabelShiftX{4}
\def\piLabelShiftY{0}

\def\piZOneLabelShiftX{-2}
\def\piZOneLabelShiftY{0}

\def\piZTwoLabelShiftX{2}
\def\piZTwoLabelShiftY{0}

\def\zeroLabelShiftX{0}
\def\zeroLabelShiftY{5}

\colorlet{curvegreen}{green!60!black}
\colorlet{curvemagenta}{magenta!80!black}
\colorlet{curvegray}{black!55}
\colorlet{curveblack}{black}
\colorlet{regionredlight}{red!7}
\colorlet{regionreddark}{red!50}

\tikzset{
  green curve/.style={draw=curvegreen, line width=\curveWidth pt},
  magenta curve/.style={draw=curvemagenta, line width=\curveWidth pt},
  wall curve/.style={draw=curvegray, line width=\curveWidth pt},
  spine curve/.style={draw=curveblack, line width=\curveWidth pt}
}

\newcommand{\branchsegment}[4]{%
  (#1) .. controls ($ (#1) + #3 $) and ($ (#2) + #4 $) .. (#2)%
}

\newcommand{\branchcontinue}[4]{%
  .. controls ($ (#1) + #3 $) and ($ (#2) + #4 $) .. (#2)%
}

\newcommand{\branchcontinuereverse}[4]{%
  .. controls ($ (#2) + #4 $) and ($ (#1) + #3 $) .. (#1)%
}

\newcommand{\yonepath}{%
  \branchsegment{left-top-anchor}{spine-top}
    {(\hLeftCtrlOneDX,\hLeftCtrlOneDY)}
    {(\hFamilyMidEndDX,\hFamilyMidEndDY)}
  \branchcontinue{spine-top}{spine-upper}
    {(\spineCtrlOneDX,\spineCtrlOneDY)}
    {(\spineCtrlTwoDX,\spineCtrlTwoDY)}
  \branchcontinuereverse{left-node}{spine-upper}
    {(\hLeftCtrlOneDX,\hLeftCtrlOneDY)}
    {(\hFamilyMidEndDX,\hFamilyMidEndDY)}
  \branchcontinuereverse{left-top-anchor}{left-node}
    {(\leftWallCtrlOneDX,\leftWallCtrlOneDY)}
    {(\leftWallCtrlTwoDX,\leftWallCtrlTwoDY)}
  -- cycle%
}

\newcommand{\ytwopath}{%
  \branchsegment{spine-top}{right-top-anchor}
    {(\hFamilyMidStartDX,\hFamilyMidStartDY)}
    {(\hRightCtrlTwoDX,\hRightCtrlTwoDY)}
  \branchcontinue{right-top-anchor}{right-node}
    {(\rightWallCtrlOneDX,\rightWallCtrlOneDY)}
    {(\rightWallCtrlTwoDX,\rightWallCtrlTwoDY)}
  \branchcontinuereverse{spine-lower}{right-node}
    {(\hFamilyMidStartDX,\hFamilyMidStartDY)}
    {(\hRightCtrlTwoDX,\hRightCtrlTwoDY)}
  \branchcontinuereverse{spine-upper}{spine-lower}
    {(\spineMidCtrlOneDX,\spineMidCtrlOneDY)}
    {(\spineMidCtrlTwoDX,\spineMidCtrlTwoDY)}
  \branchcontinuereverse{spine-top}{spine-upper}
    {(\spineCtrlOneDX,\spineCtrlOneDY)}
    {(\spineCtrlTwoDX,\spineCtrlTwoDY)}
  -- cycle%
}

\newcommand{\yoneprimepath}{%
  \branchsegment{left-node}{spine-upper}
    {(\hLeftCtrlOneDX,\hLeftCtrlOneDY)}
    {(\hFamilyMidEndDX,\hFamilyMidEndDY)}
  \branchcontinue{spine-upper}{spine-lower}
    {(\spineMidCtrlOneDX,\spineMidCtrlOneDY)}
    {(\spineMidCtrlTwoDX,\spineMidCtrlTwoDY)}
  \branchcontinue{spine-lower}{spine-bottom}
    {(\spineLowerCtrlOneDX,\spineLowerCtrlOneDY)}
    {(\spineLowerCtrlTwoDX,\spineLowerCtrlTwoDY)}
  \branchcontinuereverse{left-bottom-anchor}{spine-bottom}
    {(\hLeftCtrlOneDX,\hLeftCtrlOneDY)}
    {(\hFamilyMidEndDX,\hFamilyMidEndDY)}
  \branchcontinuereverse{left-node}{left-bottom-anchor}
    {(\leftWallLowerCtrlOneDX,\leftWallLowerCtrlOneDY)}
    {(\leftWallLowerCtrlTwoDX,\leftWallLowerCtrlTwoDY)}
  -- cycle%
}

\newcommand{\ytwoprimepath}{%
  \branchsegment{spine-lower}{right-node}
    {(\hFamilyMidStartDX,\hFamilyMidStartDY)}
    {(\hRightCtrlTwoDX,\hRightCtrlTwoDY)}
  \branchcontinue{right-node}{right-bottom-anchor}
    {(\rightWallLowerCtrlOneDX,\rightWallLowerCtrlOneDY)}
    {(\rightWallLowerCtrlTwoDX,\rightWallLowerCtrlTwoDY)}
  \branchcontinuereverse{spine-bottom}{right-bottom-anchor}
    {(\hFamilyMidStartDX,\hFamilyMidStartDY)}
    {(\hRightCtrlTwoDX,\hRightCtrlTwoDY)}
  \branchcontinuereverse{spine-lower}{spine-bottom}
    {(\spineLowerCtrlOneDX,\spineLowerCtrlOneDY)}
    {(\spineLowerCtrlTwoDX,\spineLowerCtrlTwoDY)}
  -- cycle%
}

\newcommand{\shaderegion}[3]{%
  \begin{scope}
    \clip #1;
    \shade[inner color=regionredlight, outer color=regionreddark]
      (#2) circle[radius=#3 mm];
  \end{scope}%
}

\newcommand{\drawbranch}[7]{%
  \path
    coordinate (__ctrlA) at ($ (#2) + #4 $)
    coordinate (__ctrlB) at ($ (#3) + #5 $);
  \path
    coordinate (__pre) at ($ (#2)!-#6!(__ctrlA) $)
    coordinate (__post) at ($ (#3)!-#7!(__ctrlB) $);
  \draw[#1] (__pre) -- (#2) .. controls (__ctrlA) and (__ctrlB) .. (#3) -- (__post);
}

\begin{scope}[shift={(\figShiftX,\figShiftY)}]

\coordinate (left-top-anchor) at (\xLeftTopAnchor,\yLeftTopAnchor);
\coordinate (left-node) at (\xLeftNode,\yLeftNode);
\coordinate (left-bottom-anchor) at (\xLeftBottomAnchor,\yLeftBottomAnchor);

\coordinate (spine-top) at (\xSpineTop,\ySpineTop);
\coordinate (spine-upper) at (\xSpineUpper,\ySpineUpper);
\coordinate (spine-lower) at (\xSpineLower,\ySpineLower);
\coordinate (spine-bottom) at (\xSpineBottom,\ySpineBottom);

\coordinate (right-top-anchor) at (\xRightTopAnchor,\yRightTopAnchor);
\coordinate (right-node) at (\xRightNode,\yRightNode);
\coordinate (right-bottom-anchor) at (\xRightBottomAnchor,\yRightBottomAnchor);

\coordinate (base-left-anchor) at (\xBaseLeftAnchor,\yBaseLeftAnchor);
\coordinate (base-center) at (\xBaseCenter,\yBaseCenter);
\coordinate (base-right-anchor) at (\xBaseRightAnchor,\yBaseRightAnchor);

\coordinate (spine-top-ctrl) at ($(spine-top)+(\spineCtrlOneDX,\spineCtrlOneDY)$);
\coordinate (spine-top-end) at ($(spine-top)!-\overSpineOuter!(spine-top-ctrl)$);

\coordinate (green-top-right-ctrl) at ($(right-top-anchor)+(\hRightCtrlTwoDX,\hRightCtrlTwoDY)$);
\coordinate (green-top-right-end) at ($(right-top-anchor)!-\overHorizOuter!(green-top-right-ctrl)$);
\coordinate (green-bottom-right-ctrl) at ($(right-bottom-anchor)+(\hRightCtrlTwoDX,\hRightCtrlTwoDY)$);
\coordinate (green-bottom-right-end) at ($(right-bottom-anchor)!-\overHorizOuter!(green-bottom-right-ctrl)$);

\coordinate (magenta-left-ctrl) at ($(left-node)+(\hLeftCtrlOneDX,\hLeftCtrlOneDY)$);
\coordinate (magenta-left-end) at ($(left-node)!-\overMagentaOuter!(magenta-left-ctrl)$);
\coordinate (magenta-right-ctrl) at ($(right-node)+(\hRightCtrlTwoDX,\hRightCtrlTwoDY)$);
\coordinate (magenta-right-end) at ($(right-node)!-\overMagentaOuter!(magenta-right-ctrl)$);

\coordinate (base-left-ctrl) at ($(base-left-anchor)+(\hLeftCtrlOneDX,\hLeftCtrlOneDY)$);
\coordinate (base-left-end) at ($(base-left-anchor)!-\overBaseOuter!(base-left-ctrl)$);
\coordinate (base-right-ctrl) at ($(base-right-anchor)+(\hRightCtrlTwoDX,\hRightCtrlTwoDY)$);
\coordinate (base-right-end) at ($(base-right-anchor)!-\overBaseOuter!(base-right-ctrl)$);

\coordinate (arrow-start) at (\xArrow,\yArrowStart);
\coordinate (arrow-end) at (\xArrow,\yArrowEnd);
\coordinate (arrow-mid) at ($(arrow-start)!0.5!(arrow-end)$);

\coordinate (xtotal-center) at ($(spine-top-end)+(\xZeroLabelShiftX,\xZeroLabelShiftY)$);
\coordinate (p-label) at ($(green-top-right-end)+(\pLabelShiftX,\pLabelShiftY)$);
\coordinate (q-label) at ($(green-bottom-right-end)+(\qLabelShiftX,\qLabelShiftY)$);
\coordinate (zone-label) at ($(magenta-left-end)+(\zOneLabelShiftX,\zOneLabelShiftY)$);
\coordinate (ztwo-label) at ($(magenta-right-end)+(\zTwoLabelShiftX,\zTwoLabelShiftY)$);
\coordinate (pi-label) at ($(arrow-mid)+(\piLabelShiftX,\piLabelShiftY)$);
\coordinate (pizone-label) at ($(base-left-end)+(\piZOneLabelShiftX,\piZOneLabelShiftY)$);
\coordinate (piztwo-label) at ($(base-right-end)+(\piZTwoLabelShiftX,\piZTwoLabelShiftY)$);
\coordinate (zero-label) at ($(base-center)+(\zeroLabelShiftX,\zeroLabelShiftY)$);

\coordinate (yone-top-mid) at ($(left-top-anchor)!0.5!(spine-top)$);
\coordinate (yone-bottom-mid) at ($(left-node)!0.5!(spine-upper)$);
\coordinate (yone-center-base) at ($(yone-top-mid)!0.5!(yone-bottom-mid)$);
\coordinate (yone-center) at ($(yone-center-base)+(\yOneLabelShiftX,\yOneLabelShiftY)$);

\coordinate (ytwo-top-mid) at ($(spine-top)!0.5!(right-top-anchor)$);
\coordinate (ytwo-bottom-mid) at ($(spine-lower)!0.5!(right-node)$);
\coordinate (ytwo-center-base) at ($(ytwo-top-mid)!0.5!(ytwo-bottom-mid)$);
\coordinate (ytwo-center) at ($(ytwo-center-base)+(\yTwoLabelShiftX,\yTwoLabelShiftY)$);

\coordinate (yoneprime-top-mid) at ($(left-node)!0.5!(spine-upper)$);
\coordinate (yoneprime-bottom-mid) at ($(left-bottom-anchor)!0.5!(spine-bottom)$);
\coordinate (yoneprime-center-base) at ($(yoneprime-top-mid)!0.5!(yoneprime-bottom-mid)$);
\coordinate (yoneprime-center) at ($(yoneprime-center-base)+(\yOnePrimeLabelShiftX,\yOnePrimeLabelShiftY)$);

\coordinate (ytwoprime-top-mid) at ($(spine-lower)!0.5!(right-node)$);
\coordinate (ytwoprime-bottom-mid) at ($(spine-bottom)!0.5!(right-bottom-anchor)$);
\coordinate (ytwoprime-center-base) at ($(ytwoprime-top-mid)!0.5!(ytwoprime-bottom-mid)$);
\coordinate (ytwoprime-center) at ($(ytwoprime-center-base)+(\yTwoPrimeLabelShiftX,\yTwoPrimeLabelShiftY)$);

\shaderegion{\yonepath}{yone-center}{\yOneShadeRadius}
\shaderegion{\ytwopath}{ytwo-center}{\yTwoShadeRadius}
\shaderegion{\yoneprimepath}{yoneprime-center}{\yOnePrimeShadeRadius}
\shaderegion{\ytwoprimepath}{ytwoprime-center}{\yTwoPrimeShadeRadius}

\drawbranch{green curve}
  {left-top-anchor}{spine-top}
  {(\hLeftCtrlOneDX,\hLeftCtrlOneDY)}
  {(\hFamilyMidEndDX,\hFamilyMidEndDY)}
  {\overHorizOuter}{\overHorizInner}

\drawbranch{green curve}
  {spine-top}{right-top-anchor}
  {(\hFamilyMidStartDX,\hFamilyMidStartDY)}
  {(\hRightCtrlTwoDX,\hRightCtrlTwoDY)}
  {\overHorizInner}{\overHorizOuter}

\drawbranch{magenta curve}
  {left-node}{spine-upper}
  {(\hLeftCtrlOneDX,\hLeftCtrlOneDY)}
  {(\hFamilyMidEndDX,\hFamilyMidEndDY)}
  {\overMagentaOuter}{\overHorizInner}

\drawbranch{magenta curve}
  {spine-lower}{right-node}
  {(\hFamilyMidStartDX,\hFamilyMidStartDY)}
  {(\hRightCtrlTwoDX,\hRightCtrlTwoDY)}
  {\overHorizInner}{\overMagentaOuter}

\drawbranch{green curve}
  {left-bottom-anchor}{spine-bottom}
  {(\hLeftCtrlOneDX,\hLeftCtrlOneDY)}
  {(\hFamilyMidEndDX,\hFamilyMidEndDY)}
  {\overHorizOuter}{\overHorizInner}

\drawbranch{green curve}
  {spine-bottom}{right-bottom-anchor}
  {(\hFamilyMidStartDX,\hFamilyMidStartDY)}
  {(\hRightCtrlTwoDX,\hRightCtrlTwoDY)}
  {\overHorizInner}{\overHorizOuter}

\drawbranch{wall curve}
  {left-top-anchor}{left-node}
  {(\leftWallCtrlOneDX,\leftWallCtrlOneDY)}
  {(\leftWallCtrlTwoDX,\leftWallCtrlTwoDY)}
  {\overWallOuter}{\overWallInner}

\drawbranch{wall curve}
  {left-node}{left-bottom-anchor}
  {(\leftWallLowerCtrlOneDX,\leftWallLowerCtrlOneDY)}
  {(\leftWallLowerCtrlTwoDX,\leftWallLowerCtrlTwoDY)}
  {\overWallInner}{\overWallOuter}

\drawbranch{spine curve}
  {spine-top}{spine-upper}
  {(\spineCtrlOneDX,\spineCtrlOneDY)}
  {(\spineCtrlTwoDX,\spineCtrlTwoDY)}
  {\overSpineOuter}{\overSpineInner}

\drawbranch{spine curve}
  {spine-upper}{spine-lower}
  {(\spineMidCtrlOneDX,\spineMidCtrlOneDY)}
  {(\spineMidCtrlTwoDX,\spineMidCtrlTwoDY)}
  {\overSpineInner}{\overSpineInner}

\drawbranch{spine curve}
  {spine-lower}{spine-bottom}
  {(\spineLowerCtrlOneDX,\spineLowerCtrlOneDY)}
  {(\spineLowerCtrlTwoDX,\spineLowerCtrlTwoDY)}
  {\overSpineInner}{\overSpineOuter}

\drawbranch{wall curve}
  {right-top-anchor}{right-node}
  {(\rightWallCtrlOneDX,\rightWallCtrlOneDY)}
  {(\rightWallCtrlTwoDX,\rightWallCtrlTwoDY)}
  {\overWallOuter}{\overWallInner}

\drawbranch{wall curve}
  {right-node}{right-bottom-anchor}
  {(\rightWallLowerCtrlOneDX,\rightWallLowerCtrlOneDY)}
  {(\rightWallLowerCtrlTwoDX,\rightWallLowerCtrlTwoDY)}
  {\overWallInner}{\overWallOuter}

\draw[->, line width=\arrowWidth pt] (\xArrow,\yArrowStart) -- (\xArrow,\yArrowEnd);

\drawbranch{magenta curve}
  {base-left-anchor}{base-center}
  {(\hLeftCtrlOneDX,\hLeftCtrlOneDY)}
  {(\hFamilyMidEndDX,\hFamilyMidEndDY)}
  {\overBaseOuter}{\overBaseInner}

\drawbranch{magenta curve}
  {base-center}{base-right-anchor}
  {(\hFamilyMidStartDX,\hFamilyMidStartDY)}
  {(\hRightCtrlTwoDX,\hRightCtrlTwoDY)}
  {\overBaseInner}{\overBaseOuter}

\node[anchor=south] at (xtotal-center) {$X_0$};
\node[anchor=west] at (p-label) {$p$};
\node[anchor=west] at (q-label) {$q$};
\node[anchor=east] at (zone-label) {$Z_1$};
\node[anchor=west] at (ztwo-label) {$Z_2$};
\node[anchor=west] at (pi-label) {$\pi$};
\node[inner sep=1pt] at (yone-center) {$Y_1$};
\node[inner sep=1pt] at (ytwo-center) {$Y_2$};
\node[inner sep=1pt] at (yoneprime-center) {$Y_1'$};
\node[inner sep=1pt] at (ytwoprime-center) {$Y_2'$};
\node[anchor=east] at (pizone-label) {$\pi(Z_1)$};
\node[anchor=west] at (piztwo-label) {$\pi(Z_2)$};
\node at (zero-label) {$0$};

\fill[curveblack] (base-center) circle (\dotRadius);

\end{scope}
\end{tikzpicture}